\documentclass[11pt,reqno]{amsart}
\usepackage{amsmath,amsthm,amssymb,bm}
\usepackage[dvips]{graphicx,color}
\usepackage{latexsym,amssymb}
\usepackage{mathrsfs}
\usepackage[abbrev]{amsrefs}
\usepackage{color}
\usepackage{mathtools}
\mathtoolsset{showonlyrefs=true}

\newtheorem{thm}{Theorem}[section]

\newtheorem{prop}[thm]{Proposition}
\newtheorem{lem}[thm]{Lemma}
\newtheorem{rem}[thm]{Remark}

\newcommand{\re}{\mathbb R} 

\def \Om{\Omega}
\def \pt{\partial}

\def \ep{\varepsilon}

\def \sg{\sigma}

\DeclareMathOperator{\diver}{div}
\DeclareMathOperator{\rot}{rot}

\newenvironment{pf}
{{\noindent \bf Proof. }}{\hfill $\Box$}

\numberwithin{equation}{section}
\numberwithin{thm}{section}

\pagebreak

\title[Euler equations in a critical space]{
Incompressible Euler equations in 3D bounded domains in a critical space \\
} 
\author{Tsukasa Iwabuchi and Hideo Kozono} 
\address[T. Iwabuchi]{Mathematical Institute,
Tohoku University, Sendai 980--8587, Japan, }
\email[T. Iwabuchi]{t-iwabuchi@tohoku.ac.jp}
\address[H. Kozono]{Department of Mathematics, Faculty of Science and Engineering,
Waseda University, Tokyo 169--8555, Japan, 
Mathematical Research Center for Co-creative Society, Tohoku University, 
Sendai 980-8578, Japan}
\email[H. Kozono]{kozono@waseda.jp, hideokozono@tohoku.ac.jp}

\begin{document}
\maketitle	
\begin{abstract}
We consider the 3D incompressible Euler equations in bounded domains $\Omega$ 
with smooth boundary $\partial\Omega$. Based on \cite{IMT-2019}, 
we define the Besov space $B^s_{p, q}(A)$  by means of the Stokes operator $A$ with 
the Neumann boundary condition on $\partial\Omega$, and prove unique local existence theorem of strong solution for the initial data in the critical Besov space $B^{\frac52}_{2, 1}(A)$.  
Our proof relies on the method of vanishing viscosity.  
The commutator estimate plays an essential role for derivation of energy bounds
which hold uniformly with respect to viscosity constants.       
\end{abstract}

	\section{Introduction}
\footnote[0]{2010 Mathematics Subject Classification: 35Q35; 35Q86} 
Let $\Omega$ be a bounded domain in $\re^3$	with the smooth boundary $\pt\Omega$. 
In $\Omega$ we consider the following initial-boundary value problem of the Euler equations 
for the incomressible perfect fluid;   
\begin{equation}\label{eq:1}
		\begin{cases}
			\partial_t u + (u \cdot \nabla) u + \nabla \pi = 0 
			& \mbox{in $\Omega\times (0, T)$}, 
			\\
			{\rm div \, } u = 0 
			&\mbox{in $\Omega\times (0, T)$}, 
			\\
			u \cdot \nu = 0
			&\mbox{on $\partial\Omega\times (0, T)$}, 
			\\
			u(x, 0)= u_0 (x) 
			&\mbox{in $\Omega$},  
			\end{cases}
         \end{equation}
where $u=u(x, t) = (u_1(x, t), u_2(x, t), u_3(x, t))$ and $\pi=\pi(x, t)$ denote the 
unknown velocity vector and pressure of the fluid at the point 
$x=(x_1, x_2, x_3)\in \Omega$ and the time $t\in (0, T)$, respectively, 
while $u_0=u_0(x)=(u_{0,1}(x), u_{0, 2}(x), u_{0, 3}(x))$ is the given initial velocity vector. 
Here $\nu = (\nu_1, \nu_2, \nu_3)$ denotes the unit outer normal to $\partial\Omega$. 
\par    
The purpose of this paper is to prove local existence theorem of  
classical solution $u$ to \eqref{eq:1} for arbitrary $u_0$ in a critical Besov 
space on $\Omega$. 
Temam \cite{Te-1975} proved local existence of strong solutions of \eqref{eq:1} 
with  $u_0 \in W^{m, p}(\Omega)$ for $1 <p< \infty $ and an integer $m > 1+ 3/p$.  
In the higher dimensional case, the corresponding result was obtained by 
Bourguignon-Brezis \cite{BoBr-1974}.   
Kato-Lai \cite{KaLa-1984} handled the case $\Omega \subset \re^n$ and established 
an existence theorem of local solutions to abstract evolution equations 
including \eqref{eq:1} with 
$u_0 \in H^{m}(\Omega)$ for an integer $m \ge 2 + [n/2]$.  
On the other hand, in the case when $\Omega = \re^n$, there are a number of literatures 
in the frame work of Besov spaces $B^s_{p, q}(\re^n)$ and Triebel-Lisorkin spaces 
$F^s_{p, q}(\re^n)$.   
Vishik \cite{Vi-1999} constructed a local solution with the initial vorticity 
$\rot u_0\in B^s_{p, q}(\re^n)$ for $1 <p, q < \infty$ and $s > n/p$.  
Chae \cite{Ch-2002} obtained the corresponding result to the initial data 
$u_0 \in F^s_{p, q}(\re^n)$ for $1 <p, q < \infty$ and $s > 1 + n/p$.  
Pak-Park \cite{PaPa-2004} has been successful to obtain the local solution  
in such a critical Besov space as $B^1_{\infty, 1}(\re^n)$.  
\par
In the present paper, we prove the local well-posedness in the critical Besov space
$B^{\frac52}_{2, 1}(\Omega)$.  Based on \cite {IwKo-preprint}
 and \cite{IMT-2019}, we first define the Besov space 
$B^s_{p, q}(A)$ in terms of the Stokes operator $A$ in $L^2_\sigma(\Omega)$ 
with the Neumann boundary condition on $\pt\Om$.  
Our method relies on the invisid limit 
$\displaystyle{\lim_{\varepsilon \to +0} u_{\varepsilon}}$ 
of a family $\{u_\varepsilon\}_{\varepsilon >0}$ of solutions to the following Navier-Stokes equations; 
\begin{equation}\label{eq:1-2}
		\begin{cases}
			\partial_t u_{\varepsilon} - \varepsilon\Delta u_{\varepsilon} + 
			(u_{\varepsilon} \cdot \nabla) u_{\varepsilon} + \nabla \pi_{\varepsilon} = 0 
			& \mbox{in $\Omega\times (0, T)$}, 
			\\
			{\rm div \, } u_{\varepsilon} = 0 
			&\mbox{in $\Omega\times (0, T)$}, 
			\\
			u_{\varepsilon} \cdot \nu = 0, \,\,\, \rot u_{\varepsilon} \times \nu =0
			&\mbox{on $\partial\Omega\times (0, T)$}, 
			\\
			u_{\varepsilon}(x, 0)= u_0 (x) 
			&\mbox{in $\Omega$}.   
			\end{cases}
\end{equation}
It is essential to establish uniform bounds of $\{u_\ep\}_{\ep>0}$ 
in the Chemin-Lernar space $\tilde L^{\infty}(0, T; B^{\frac52}_{2, 1}(A))$. 
The commutator type estimate such as Kato-Ponce \cite{KaPo-1988}  
plays a crucial role for such bounds.  
Since we make use of the Stokes operator $A$ with the Neumann boundary condition 
like \eqref{eq:1-2}, the boundary layer does not appear in the invisid limit. 
\par 
\vspace{1mm}
To state our result, let us recall the Stokes operator $A_p$ on $L^p_\sg(\Omega)$ 
for $1 < p < \infty$.  We denote by $C^\infty_{0, \sg}(\Omega)$ the 
set of $C^\infty$-solenoidal vector functions with compact support in $\Omega$.  
$L^p_\sg(\Om)$ is the closure of $C^\infty_{0, \sg}(\Omega)$ with respect to the 
$L^p$-norm $\|\cdot\|_{L^p}$ on $\Omega$. Then the Helmholtz projection 
$\mathbb{P}_p$ from $L^p(\Omega)$ onto $L^p_{\sg}(\Om)$ is well-defined. 
We define $A_p$ by 
\begin{equation*}
\left\{
\begin{array}{ll}
& D(A_p)\equiv\{ u \in W^{2, p}(\Omega)|\,\, u\cdot \nu =0, \, \, \, \rot u\times \nu =0 
\,\,\, \mbox{on $\pt\Om$}\}\cap L^p_\sg(\Omega), \\
& A_pu \equiv -\mathbb{P}_p\Delta u = -\Delta u \quad \mbox{for $u \in D(A_p)$}. 
\end{array}
\right.
\end{equation*}
Concerning the kernel $\mbox{Ker}A_p$ of $A_p$, it is known (see e.g., \cite[Lemma 2.2]{KoShYa-2025}) that 
$$
\mbox{Ker}A_p =\{h \in C^{\infty}(\bar \Omega);\,  \diver h=0, \, \, \rot h =0, \,\,\,
h\cdot \nu|_{\pt\Om}= 0\} 
\quad
\mbox{for all $1 < p < \infty$}.
$$
Hence, in general $A_2$ is a non-negative self-adjoint operator in $L^2_\sg(\Omega)$ 
with zero eigenvalue, and is strictly positive if and only if 
$\Omega$ is a simply connected domain in $\re^3$.  
However, by considering $\tilde A_2 = A_2 + 1$ instead of $A_2$ itself, 
we may regard $A_2$ as a strictly positive self-adjoint operator in $L^2_\sg(\Omega)$ with bounded inverse $A_2^{-1}$ 
in the proof of our main result.   For such a correction, it suffices to replace 
$-\ep \Delta u_\ep$ by $\ep(-\Delta +1)u_\ep$ in \eqref{eq:1-2}.        
\par 
Let us define Besov spaces $B^s_{p, q}(A)$ in terms of $A=A_2$.  
Take a function $\phi \in C^\infty_0(-\infty, \infty)$ in such a way that 
	\[   
	{\rm supp \,} \phi \in [2^{-1},2] \quad
	0 \leq \phi \leq 1 , \quad 
	\sum_{j \in \mathbb Z} \phi (2^{-j}\lambda ) = 1, \quad \lambda >0. 
	\]
	We introduce $\{ \phi_j \}_{j \in \mathbb Z} $ such that for every $j \in \mathbb Z$
	\begin{equation}\label{0116-3}
	\phi_j(\lambda) := \phi(2^{-j} \lambda), \quad \lambda \in \mathbb R, 
\end{equation}
	which is a partition of the unity by dyadic numbers 
	on the half line $(0,\infty)$. 
	For $1 \le p, q \le \infty$ and $s \in \re$,
	the norm  of $B^s_{p, q}(A)$ is given by 
	\begin{equation}\label{0116-1}
\| u \|_{B^s_{p,q}(A)} 
:=  \Big\| \Big(1- \sum _{j \geq 1} \phi_j(\sqrt{A})\Big) f \Big\|_{L^p} +  \Big\| \Big\{ 2^{js} \| \phi_j(\sqrt{A}) u \|_{L^p} \Big\}_{j \in \mathbb N} \Big\|_{\ell^q (\mathbb N)} . 
\end{equation}
\par 
Our main result reads as follows.  
\begin{thm}\label{thm:1}
For every $u_0 \in B^{\frac{5}{2}}_{2,1}(A)$ there exist a time $T > 0$ 
and a unique solution $u \in C([0,T], B^{\frac{5}{2}}_{2,1}(A)) $ of 
(\ref{eq:1}).  Such a solution $u$ satisfies 
$u \in C^1 ([0,T], W^{1,p}(\Omega))$ for all $1 \leq p \leq 3$. 
\end{thm}
\par
\begin{rem} 
{\rm (i)} The above theorem shows unique existence of (\ref{eq:1}) in the critical Besov space 
$B^{\frac{5}{2}}_{2,1}(A)$ in any bounded domain $\Omega \subset \re^3$ with smooth boundary 
$\pt\Om$.  It should be emphasized that any geometric restriction on $\Omega$ 
is unnecessary.   
\par
{\rm (ii)} Obviously, we see that the corresponding result to Theorem \ref{thm:1} holds 
in $\re^3$.  
Bourgain and Li~\cite[Theorem 1.7]{BourLi-2021} 
considered the Cauchy problem of the Euler equations in $\mathbb R^n$ for $n=2,3$, 
and proved ill-posedness in $B^{1+ \frac np}_{p, q}(\re^n)$ for $1 < p < \infty, 1 < q \leq \infty$.   Hence, our result may be regarded as an optimal theorem in $\mathbb R^3$ which exhibits the borderline case between well-posedness and ill-posedness in Besov spaces. 
\par
{\rm (iii)} 
Higher regularity yields a better approximation rate.  Specifically, for $s = 5/2$ and $u \in \dot B^{\frac{5}{2}}_{2,1}(A)$, there exists a constant $C > 0$ such that
\[
\left\| u - \sum_{j \leq J} \phi_j (\sqrt{A}) u \right\|_{L^2} 
\leq C 2^{-\frac{5}{2} J}
\]
for all $J \in \mathbb{N}$. 
Furthermore, since the spectrum of the Stokes operator on a bounded domain consists only of eigenvalues, the partial sum above is a finite linear combination of the corresponding eigenfunctions.
\end{rem}

This paper is organized as follows. 
In Section 2, we provide several preliminary lemmas and propositions required to prove Theorem~\ref{thm:1}. 
Section 3 is devoted to constructing solutions for the equation regularized by $\varepsilon \Delta$. 
In Section 4, we establish the uniform boundedness of the solutions obtained in Section 3 with respect to the parameter $\varepsilon > 0$. 
Finally, in Section 5, we provide the proof of Theorem~\ref{thm:1} by taking the limit as $\varepsilon \to 0$. 
Throughout the following sections, we suppose that $A_2$ is strictly positive.

	\section{Preliminary}

	In this section, we collect several preliminary definitions and results. In Subsection~2.1, we fundamental properties of the Stokes operator $A$ 
and the associated Besov spaces $B^s_{p,q}(A)$. 
Subsection~2.2 is devoted to providing several propositions and lemmas required for the proof of Theorem~\ref{thm:1}, including Sobolev-type embeddings, derivative estimates, and various properties of derivatives involving spectral restriction operators.

\subsection{The Stokes operator and Besov spaces} 
We begin by recalling several fundamental results regarding vector fields, the Laplacian, and the Stokes operator 
(see~\cite{Miy-1980}). Following the procedure established in \cite{IMT-2019, FI-2024}, we introduce Besov and Sobolev spaces associated with the Stokes operator.

Let $1 < p < \infty$. 
The Laplacian $B_p$ with the Neumann condition acting on vector fields in $L^p$ is defined as follows:
$$
\begin{cases}
	D(B_p) = \{ u \in W^{2,p} (\Omega) \, | \, {\rm rot} \, u \times \nu = 0 , u \cdot \nu = 0 \text{ on } \partial \Omega \}, 
	\\
	B_p u = - \Delta u, \quad u \in D(B_p). 
\end{cases}
$$
For the Stokes operator $A_p$ on $L^p_\sigma$ defined in the previous section, 
we have $D(A_p)= \mathbb{P}_pD(B_p)$, and it holds that 
\[
B_p\mathbb{P}_p = \mathbb{P}_pB = A_p\mathbb{P}_p 
\quad
\mbox{on $D(B_p)$}.  
\]
We here explain the derivation of the projection $\mathbb P_p$ briefly.  
In the $L^p$ setting, the solvability of the following Neumann problem is well-established:
\begin{equation}\label{0115-1}
	\begin{cases}
		\Delta \pi = \mathrm{div} \, u & \text{in } \Omega, \\
		\dfrac{\partial \pi}{\partial \nu} = u \cdot \nu & \text{on } \partial \Omega.
	\end{cases}
\end{equation}
More precisely, we consider the weak formulation of the Neumann problem. It is known that for a given vector field $u \in L^p$, there exists a unique gradient $\nabla \pi \in L^p$ such that
\[
\langle \nabla \pi, \nabla \phi \rangle = \langle u, \nabla \phi \rangle \quad \text{for all } \phi \in H^1_{p'}(\Omega),
\]
where $\langle \cdot, \cdot \rangle$ denotes the dual coupling between $L^p$ and $L^{p'}$ with $1/p + 1/p' = 1$. The solvability of this problem is equivalent to the existence of a positive constant $C$ such that
\[
\| \nabla \pi \|_{L^p} \leq C \sup_{\phi \in H^1_{p'}(\Omega) \setminus \{0\}} \frac{|\langle \nabla \pi, \nabla \phi \rangle|}{\| \nabla \phi \|_{L^{p'}}}.
\]
Based on this, we define the Helmholtz projection $\mathbb{P}_p : L^p \to L^p_\sigma$ by
\[
\mathbb{P}_p u = u - \nabla \pi,
\]
where $\nabla \pi$ is the gradient of the solution to the weak Neumann problem described above. 
Furthermore, we note the following identity:
\[
\Delta \nabla \pi = \nabla \mathrm{div} \, u.
\]
For a smooth solution $(u,\pi)$ of the Euler equations \eqref{eq:1}, this identity yields
\begin{equation}\label{1225-1}
	\Delta \Big( (u \cdot \nabla) u + \nabla \pi \Big) = \Delta \Big( (u \cdot \nabla) u \Big) + \nabla \mathrm{div} \Big( (u \cdot \nabla) u \Big).
\end{equation}
The following is a well-known fact regarding the orthogonality between divergence-free and rotation-free vector fields.
\begin{lem} \label{lem:0910-1}
	Let $1 < p < \infty$ and $1/p + 1/p' = 1$. For every scalar function $\pi \in W^{1,p'}(\Omega)$ and every vector field $v \in L^p(\Omega)$, we have
	\[
	\int_{\Omega} \nabla \pi \cdot \mathbb{P} v \, dx = 0.
	\]
\end{lem}

\begin{proof}
	For any $v \in L^p(\Omega)$, its Helmholtz projection satisfies $\mathbb{P}v \in L^p_\sigma(\Omega)$. By the density of $C^\infty_{0,\sigma}(\Omega)$ in $L^p_\sigma(\Omega)$ for $1 < p < \infty$, there exists a sequence $\{ u_N \}_{N=1}^\infty \subset C^\infty_{0,\sigma}(\Omega)$ such that $u_N \to \mathbb{P}v$ in $L^p(\Omega)$ as $N \to \infty$. 
	For $ \pi \in W^{1,p}(\Omega)$, it is straightforward that 
	\[
	\int _{\Omega} \nabla p \cdot u_N~dx =0 ,
	\]
	and the equality for $\mathbb P v$ instead of $u_n$ is ensured by taking the limit as $N \to \infty$. 
\end{proof}

The operator $A_2$ is a non-negative self-adjoint operator on $L^2_\sigma$. By the spectral theorem, there exists a resolution of the identity $\{ E(\lambda) \}_{\lambda \in \mathbb R}$ such that 
\[ 
u = \int_{-\infty}^\infty dE(\lambda) u \text{ in } L ^2 _\sigma \quad \text{for all } u \in L^2_\sigma.
\]
For any measurable function $\varphi : \mathbb R \to \mathbb C$, the operator $\varphi(A_2)$ is well-defined on $L^2_\sigma$ as
\[ 
\varphi(A_2) u = \int_{-\infty}^\infty \varphi(\lambda) dE(\lambda ) u.
\]
The following proposition is fundamental for extending the framework of Besov and Sobolev spaces from $L^2$ to the $L^p$ setting.

\vskip3mm

\begin{prop}\label{prop:0114-1}
	Let $1 \leq p \leq \infty$, $\phi \in C_0 ^\infty (\mathbb R)$, and $\phi_j(\lambda) = \phi(2^{-2j} \lambda)$ for $j \in \mathbb Z$ and $\lambda \in \mathbb R$. Then 
	\[
	\sup _{j \in \mathbb Z} \| \phi_j(\sqrt{A_2}) \|_{L^p \to L^p} < \infty . 
	\]
\end{prop}

The proof of Proposition~\ref{prop:0114-1} is analogous to that in \cite{IMT-2018} (see also Subsection~2.2 in \cite{IwKo-preprint}). Next, we introduce the space of test functions and its dual space.

\vskip2mm

\noindent
{\bf Definition. } (Spaces of test functions and distributions) 
\begin{enumerate}
	\item We define the space of test functions $\mathcal{X}_\sigma$ by
	\[
	\mathcal{X}_\sigma := \{ u \in L^1 \cap L^2_\sigma \mid p_M(u) < \infty \text{ for all } M \in \mathbb{N} \},
	\]
	where the semi-norms are given by
	\[ 
	p_M(u) := \| u \|_{L^1} + \sup _{j \in \mathbb{N}} 2^{Mj} \| \phi_j(A_2) u \|_{L^1}.
	\]
	\item We denote by $\mathcal{X}_\sigma '$ the topological dual space of $\mathcal{X}_\sigma$.
\end{enumerate}

\vskip2mm 

The spaces $\mathcal{X}_\sigma$ and $\mathcal{X}'_\sigma$ are of non-homogeneous type. Since we may suppose that  the operator $A_2$ is strictly positive (i.e., the infimum of its spectrum is strictly positive), the non-homogeneous and homogeneous spaces are equivalent. Here, we adopt the non-homogeneous framework.

We introduce the projection $\mathbb{P}$ onto the space of divergence-free vector fields and the Stokes operator on $\mathcal{X}_\sigma'$. While the operator $A$ and its resolvent act on divergence-free fields, $B$ and its resolvent act on a more general class of vector fields.

\vskip3mm

\noindent
{\bf Definition. } (The projection $\mathbb{P}$, the Stokes operator $A$ and the Laplacian $B$ on $\mathcal{X}_\sigma'$)
\begin{enumerate}
	\item For $u \in L^p$ with $1 \leq p \leq \infty$, we define $\mathbb{P} u \in \mathcal{X}_\sigma'$ by
	\[
	{}_{\mathcal{X}_\sigma'} \langle \mathbb{P} u, v \rangle_{\mathcal{X}_\sigma} := \int_{\Omega} u \cdot \mathbb{P}_2 v \, dx \quad \text{for } v \in \mathcal{X}_\sigma.
	\]
	Hereafter, we simply write $\mathbb{P}$ for $\mathbb{P}_2$.
	
	\item For $u \in \mathcal{X}_\sigma'$, we define $Au \in \mathcal{X}_\sigma'$ by
	\[
	{}_{\mathcal{X}_\sigma'} \langle Au, v \rangle_{\mathcal{X}_\sigma} := {}_{\mathcal{X}_\sigma'} \langle u, A_2 v \rangle_{\mathcal{X}_\sigma} \quad \text{for } v \in \mathcal{X}_\sigma.
	\]
	
	\item Let $\phi \in C^\infty(\mathbb R) \cap L^\infty(\mathbb{R})$. For $u \in \mathcal{X}_\sigma'$, we define $\phi(A) u \in \mathcal{X}_\sigma'$ by
	\[
	{}_{\mathcal{X}_\sigma'} \langle \phi(A) u, v \rangle_{\mathcal{X}_\sigma} := {}_{\mathcal{X}_\sigma'} \langle u, \phi(A_2) v \rangle_{\mathcal{X}_\sigma} \quad \text{for } v \in \mathcal{X}_\sigma.
	\]
	We write $A$ for $A_2$ when the context is clear.
	
		\item Similarly, we define the operators $B$ and $\phi(B)$ on $\mathcal{X}_\sigma'$ as the dual operators of $B_2$ and $\phi(B_2)$, respectively. That is, for $u \in \mathcal{X}_\sigma'$,
	\[
	{}_{\mathcal{X}_\sigma'} \langle Bu, v \rangle_{\mathcal{X}_\sigma} := {}_{\mathcal{X}_\sigma'} \langle u, B_2 v \rangle_{\mathcal{X}_\sigma}, \quad 
	{}_{\mathcal{X}_\sigma'} \langle \phi(B) u, v \rangle_{\mathcal{X}_\sigma} := {}_{\mathcal{X}_\sigma'} \langle u, \phi(B_2) v \rangle_{\mathcal{X}_\sigma}
	\]
	for all $v \in \mathcal{X}_\sigma$. We also abbreviate $B_2$ as $B$.
\end{enumerate}

\vskip2mm 

From the definition above, it follows that
\[
f = \mathbb{P} f \quad \text{in } \mathcal{X}'_\sigma,
\]
for any $f \in \mathcal{X}'_\sigma \cap L^p$ with $1 \leq p \leq \infty$. 

The $L^p$-boundedness of the Helmholtz projection is well-established for $1 < p < \infty$ (see, e.g., \cite[Section 2]{Miy-1980}). Since our operator $\mathbb{P}$ is defined as a natural extension to the distribution space $\mathcal{X}'_\sigma$, its boundedness on $L^p$ follows immediately. 

\begin{lem} \label{lem:1222-2} 
	Let $1 < p < \infty$. Then $\mathbb{P}$ is a bounded linear operator on $L^p$. 
\end{lem}

The following lemma shows that the resolvent of $A$ can be expressed using the resolvent of $B$ and the projection $\mathbb{P}$. 

\begin{lem} 
	Let $1 \leq p \leq \infty$ and $\mu \geq 0$. The operator $(\mu + A)^{-1} \mathbb{P}: L^p \to \mathcal{X}_\sigma'$ is well-defined and satisfies 
	\begin{equation} \label{1219-1} 
		(\mu+B)^{-1} \mathbb{P} = \mathbb{P} (\mu+B)^{-1} = (\mu + A)^{-1} \mathbb{P}. 
	\end{equation}
\end{lem}

\noindent 
{\bf Remark. } 
This lemma is intended to be applied to the nonlinear term $(u \cdot \nabla) u$, which in general does not satisfy the divergence-free condition or the prescribed boundary conditions. The identity \eqref{1219-1} also holds as an operator identity on $\mathcal{X}'_\sigma$. 

\vskip2mm 

\begin{pf}
	It is shown in \cite[Corollary~3.6]{Miy-1980} that the identity
	\[
	(\mu+B_p)^{-1} \mathbb{P}_p = \mathbb{P}_p (\mu+B_p)^{-1} = (\mu + A_p)^{-1} \mathbb{P}_p
	\]
	holds for any $\mu$ in the resolvent set of $B_p$ within the $L^p$ framework for $1 < p < \infty$. By the definition of the projection $\mathbb{P} : L^p\to \mathcal{X}'_\sigma$ and the duality argument, we obtain the well-definedness and the desired identity \eqref{1219-1}.
\end{pf}

\vskip2mm

We now define the Besov and Sobolev spaces associated with the Stokes operator $A$.

\vskip3mm

\noindent
{\bf Definition. } (Besov and Sobolev spaces) \quad 
Let $s \in \mathbb{R}$ and $1 \leq p, q \leq \infty$. We define the Besov space $B^s_{p,q}(A)$ associated with the Stokes operator $A$ by
\[
B^s_{p,q}(A) := \left\{ u \in \mathcal{X}_\sigma' \Bigm| \| u \|_{B^s_{p,q}(A)} := \left\| \left\{ 2^{sj} \| \phi_j(\sqrt{A}) u \|_{L^p} \right\}_{j \in \mathbb{Z}} \right\|_{\ell^q(\mathbb{Z})} < \infty \right\},
\]
where $\{ \phi_j \}_{j \in \mathbb{Z}}$ is a dyadic partition of unity satisfying \eqref{0116-3}. 
Similarly, for $s \in \mathbb{R}$, we define the Sobolev space $H^s(A)$ by
\[
H^s(A) := \left\{ u \in \mathcal{X}_\sigma' \Bigm| \| u \|_{H^s(A)} := \| A^{\frac{s}{2}} u \|_{L^2} < \infty \right\}.
\]

\vskip2mm 

\noindent 
{\bf Remark.} In the definition above, we adopt the norm of homogeneous type as an equivalent norm for $B^s_{p,q}(A)$. We note that the homogeneous and non-homogeneous Besov spaces are equivalent under the present assumptions on $\Omega$, since the infimum of the spectrum of the Stokes operator $A$ is strictly positive. Consequently, the non-homogeneous norm introduced in \eqref{0116-1} is equivalent to the homogeneous one as follows:
\begin{equation} \notag 
\begin{split}
	&\left\| \left\{ 2^{js} \| \phi_j(\sqrt{A}) u \|_{L^p} \right\}_{j \in \mathbb{Z}} \right\|_{\ell^q (\mathbb{Z})}.
\\
		&
		\simeq 
		\left\| \left( 1 - \sum_{j=1}^\infty \phi_j(\sqrt{A}) \right) u \right\|_{L^p} + \left\| \left\{ 2^{js} \| \phi_j(\sqrt{A}) u \|_{L^p} \right\}_{j=1}^\infty \right\|_{\ell^q (\mathbb{N})} . 
\end{split}
\end{equation}

\vskip2mm

\noindent 
{\bf Definition.} We denote $D := \sqrt{A}, \sqrt{B}$.

\vskip2mm 

\noindent 
{\bf Remark.} We frequently employ the notation $D$ for spectral restriction operators instead of $\sqrt{A}$ or $\sqrt{B}$. This is justified by the fact that
\[
\phi_j(\sqrt{A}) u = \phi_j(\sqrt{B}) \mathbb{P} u = \phi_j(\sqrt{B}) u \quad \text{in } L^p_\sigma
\]
for $u \in L^p_\sigma$. The divergence-free property is crucial for this identity. Indeed, if the divergence-free condition is not assumed, $\phi_j(\sqrt{B}) \mathbb{P} u$ remains well-defined for any $u \in L^p$, whereas $\phi_j(\sqrt{A}) u$ is generally not defined for arbitrary vector fields. There is an essential difference in the boundary conditions inherent in the definitions of $A$ and $B$; specifically, operators associated with $A$ act on vector fields that intrinsically satisfy the divergence-free condition. In our subsequent analysis, the nonlinear term $(u \cdot \nabla) u$ will be handled by applying the projection $\mathbb{P}$.

\subsection{Analytical tools}

In this subsection, we provide several analytical tools required for our main results. These include elliptic estimates involving weak derivatives and the operators $A$ and $B$, as well as fundamental properties of spectral restriction operators and the Stokes semigroup.

\begin{lem} \label{lem:1222-8}
	Let $1 < p < \infty$.
	\begin{enumerate}
		\item {\rm(}Elliptic estimates{\rm)} There exists a constant $C > 0$ such that for all $u \in L^p_\sigma$,
		\[
		\| \nabla^2 A^{-1} u \|_{L^p} \leq C \| u \|_{L^p}.
		\]
		\item There exists a constant $C > 0$ such that for all $u \in L^p_\sigma$,
		\[
		\| \nabla A^{-\frac{1}{2}} u \|_{L^p} \leq C \| u \|_{L^p}.
		\]
		\item The operators $\nabla^2 B^{-1}$ and $\nabla B^{-\frac{1}{2}}$ are bounded linear operators on $L^p$.
		
		\item For any $u \in L^2$, the identity $A^{\frac{1}{2}} \mathbb{P} u = A\mathbb{P} B^{-\frac{1}{2}} u$ holds in $\mathcal{X}'_\sigma$.
	\end{enumerate}
\end{lem}

\begin{proof}
	(1) and (2) are established in \cite[Proposition~2.1]{IwKo-preprint}. 
	(3) can be proved by an argument analogous to (1) and (2). 
For (4),
we recall the identity  $A_2 ^{\frac{1}{2}} u = A_2A_2^{-\frac{1}{2}} u $ in $L^2 $ for all $u \in H^{\frac{1}{2}}(A)$. 
We utilize the integral representation of the fractional power: 
\[ A_2 ^{-\frac{1}{2}} = c_0 \int_0^\infty t^{-\frac{1}{2}} (1+t A_2)^{-1} dt , 
\]
where $c_0$ is an appropriate positive constant. This representation, together with the commutativity relation 
\eqref{1219-1}, yields that for $u \in L^2$
\[
A_2A_2^{-\frac{1}{2}} u 
= A_2 c_0 \int_0^\infty t^{-\frac{1}{2}} (1+tA_2)^{-1} u~dt 
= A_2  \mathbb P c_0 \int_0^\infty t^{-\frac{1}{2}} (1+tB_2)^{-1} u~dt 
= A_2 \mathbb P B_2 ^{-\frac{1}{2}}u. 
\]
The identity (4) then follows from the definitions of $A$ and $B$ on $\mathcal{X}'_\sigma$. 
\end{proof}

\begin{lem}\label{lem:1222-10}
	Let $m\in \mathbb N$ and $1 < p < \infty$. Then a positive constant $C$ exists such that 
\[
\| B^{-1} u \|_{W^{m+2,p}(\Omega)} \leq C \| u \|_{W^{m,p}(\Omega)} 
\quad \text{for } 	u \in W^{m,p}(\Omega). 
\]
\end{lem}

\begin{pf}
	The case when $m=1,2$ is already established in Lemma~\ref{lem:1222-8}. 
	When $m \geq 3$, we apply Lemma~4.4 in \cite{KoYa-2009} under the boundary condition 
$u \cdot \nu  = {\rm rot \, } u \times \nu =0$. 
\end{pf}

\begin{lem}\label{lem:1222-9}
	Let $m\in \mathbb N$ and $1 < p < \infty$. Then a positive constant $C$ exists such that 
	\[
	\| \mathbb P u \|_{W^{m,p}(\Omega)} \leq C \| u \|_{W^{m,p}(\Omega)} 
	\quad \text{for } 	u \in W^{m,p}(\Omega). 
	\]
\end{lem}

\noindent {\bf Proof. }
Following the proof of Lemma 3.3 (i)  in \cite{GiMi-1985}, the above inequality is proved. 
In fact, we write $\mathbb P u = u - \nabla \pi $, where $ \pi$ is the solution of the Neumann problem \eqref{0115-1}, 
and the elliptic esimate in $W^{m,p}(\Omega)$ for $\nabla \pi$ is available. 
\hfill $\Box$

\begin{lem}\label{lem:1222-4}
	\begin{enumerate}
		\item	We have the norm equivalence that for every $u \in H^1 (A)$ 
		\[
		\| u \|_{H^1(A)} \simeq \| \nabla u \|_{L^2}. 
		\]
		Furthermore, for every $u \in L^2 $ with $\| B^{\frac{1}{2}}  u \|_{L^2} < \infty$, 
		\[
		\| B^{\frac{1}{2}} u \|_{L^2} \simeq \| \nabla u \|_{L^2}. 
		\]

		\item 
		Let $|s-2|< 1$. We have the norm equivalece that for every $u \in B^s_{2,1}(A)$
		\[ \| u \|_{B^s_{2,1}(A)} \simeq 
		\sum_{j \in \mathbb Z} 2^{(s-2)j} \|  \psi_j(D)^{\frac{1}{2}} \Delta u \|_{L^2},
		\]
		where $\psi_j$ is defined by 
	\end{enumerate}
\begin{equation} \label{0115-2}
		\psi_j(\lambda) = (1 + 2^{-2j-2} \lambda^2)^{-1} - (1 + 2^{-2j} \lambda^2)^{-1}, \quad \lambda \geq 0.
\end{equation}
\end{lem}
\noindent 
{\bf Remark.} 
The norm on the right-hand side of Lemma~\ref{lem:1222-4}~(2) does not directly control the boundary values of $u$; however, it imposes a restriction on $\Delta u$.

\vskip2mm 

\begin{pf}
	(1) Let $u \in H^1(A)$ and define $u_N = \displaystyle \sum _{j \leq N} \phi_j(D) u$. 
	It is straightforward  to see that $u_N$ converges to $u$ in $H^1(A)$ as $N \to \infty$, and that $u_N \in H^2(A)$. 
	By using $\Delta = - {\rm rot \, rot \, } + \nabla {\rm div }$, and the fact that ${\rm div \, } u_N = 0 \text{ in  } \Omega$ 
	and ${\rm rot \,} u_N \times \nu = 0 \text{ on } \partial \Omega$ , we have 
 	\[ 
	  \| u_N \|_{H^1(A)} = \int _{\Omega} (- \Delta u_N ) \cdot u_N ~dx 
	  = \int  _{\Omega} ( {\rm rot \, rot \, } u_N ) \cdot u_N ~dx 
		  =\| {\rm rot \, }u _N \|_{L^2}^2 
		  \leq C \| \nabla u_N \|_{L^2}^2 ,
	\]
	which proves one side of the inequality. 
	Conversely, we apply  in \cite[Theorem~2.4]{KoYa-2009} to obtain the reverse inequality. 
	Under the assumption that $\Omega$ excludes non-trivial harmonic fields, we have
\[
\| \nabla u_N \|_{L^2} \leq C \| {\rm rot \, } u_N \|_{L^2} + C \| {\rm div \, } u_N \|_{L^2}= C \| u_N \|_{H^1(A)},
\]
where we have used the divergence-free property of $u_N$. 
Thus $\{u_N\}_{N \in \mathbb N}$ is a bounded sequence in $H^1(\Omega)$ and converges to $u$ strongly in $L^2(\Omega)$. 
It follows that the limit $u$ belongs to $H^1(\Omega) \cap L^2 _\sigma$. 
Taking the limit in the equivalence $\| u_N \|_{H^1(A)} \simeq \| \nabla u_N  \|_{L^2}$ as $N \to \infty$, 
we conclude that 
\[
\| u\|_{H^1(A)} \simeq \| \nabla u \|_{L^2} .
\]

\noindent 
(2)  First, we observe that $\phi_j(\lambda) \psi_j(\lambda)^{-\frac{1}{2}}$ for $\lambda \in \mathbb{R}$ is a smooth function with compact support. Furthermore, it satisfies the scaling property
\[ \phi_j(\lambda ) \psi_j(\lambda) ^{-\frac{1}{2}} = \phi_0( 2^{-j}\lambda ) \psi_0( 2^{-j}\lambda) ^{-\frac{1}{2}}. 
\]
	By the spectral theorem, we have the uniform boundedness of $\phi_j(D) \psi_j(D)^{-\frac{1}{2}}$ with respect to $j \in \mathbb Z$ as operators on $L^2$. 
We observe that the operator $\phi_j(D)$ restricts the spectrum to a dyadic shell around $2^j$, where the operator $B$ is bounded from below by $C 2^{2j}$. Consequently, we have
\[
2^{2j}\| \phi_j (D)u \|_{L^2} \leq C  \| B \phi_j(D) \psi_j(D)^{-\frac{1}{2}}  \psi_j(D) u \|_{L^2} 
\leq C \| \psi_j(D) ^{\frac{1}{2}} \Delta u \|_{L^2}. 
\]
	Multiplying by $2^{(s-2)j}$ and summing over $j \in \mathbb{Z}$, we arrive at
\[
 \| u \|_{\dot B^s_{2,1}(A)} \leq C \sum_{j \in \mathbb Z} 2^{(s-2)j} \| \psi_j(D)^{\frac{1}{2} } \Delta u \|_{L^2} . 
\]

Conversely, 
we prove the reverse inequality with the decomposition by $\{ \phi_j \}_{j \in \mathbb Z}$. 
We note that 
\[
\psi_j(D)^{\frac{1}{2}} 
=\bigg( \frac{3}{4} \cdot 2^{-2j} B (1+ 2^{(-2j-2)} B )^{-1} (1+ 2^{-2j} B)^{-1} \bigg) ^{\frac{1}{2}} \sum _{k \in \mathbb Z} \phi_k(D)
\]
and for each $k$, it follows that 
\[
\begin{split}
&	\bigg\| 
\bigg( \frac{3}{4} \cdot 2^{-2j} B (1+ 2^{(-2j-2)} B )^{-1} (1+ 2^{-2j} B)^{-1} \bigg) ^{\frac{1}{2}}\phi_k(D) \Delta u
\bigg\| _{L^2} 
\\
\leq 
&C \dfrac{2^{-j + k}}{1+2^{-2j + 2k}}  \| \phi_k(D) \Delta u \|_{L^2}
\\
\leq 
&C 2^{-|j-k|} \| \phi_k(D) \Delta u \|_{L^2}. 
\end{split}
\]
Thus, we have
\[
\begin{split}
	\sum _{j \in \mathbb Z} 2^{(s-2)j} \| \psi_j(D) ^{\frac{1}{2}} \Delta u \|_{L^2}
\leq
& C \sum _{j \in \mathbb Z} \sum_{k \in \mathbb Z} 2^{(s-2)(j-k)} 2^{-|j-k|} \cdot  2^{(s-2)k} \| \phi_k(D) \Delta u \|_{L^2}
\\
\leq 
& C \Big( \sum _{j \in \mathbb Z} 2^{(s-2)j - |j|} \Big) \sum _{k \in \mathbb Z} 2^{(s-2)k} \| \phi_k(D) \Delta u \|_{L^2}
\\
\leq 
& C \| u \|_{B^s_{2,1}(A)}. 
\end{split}
\]
Since $|s-2| < 1$ ensures the convergence of the series $\sum_{m \in \mathbb{Z}} 2^{(s-2)m - |m|}$, the proof is complete.
\end{pf}

\begin{lem}
	Let $\psi_j $ be defined by \eqref{0115-2}. 
	Then for every $j \in \mathbb Z$
	\begin{equation}\label{1219-2}  \| \psi_j(D) \|_{L^2 \to L^2} \leq \| \psi_j(D)^{\frac{1}{2}} \|_{L^2 \to L^2}. 
\end{equation}
A positive constant $C$ exists such that for every $j \in \mathbb Z$
	\begin{equation}\label{1219-3}  \| \nabla (1+2^{-2j}B)^{-1} \|_{L^2 \to L^2} \leq  C 2^{j} \| \psi_j(D)^{\frac{1}{2}} \|_{L^2 \to L^2}. 
\end{equation}
\end{lem}

\begin{pf}
	By the definition of $\psi_j(D)$,
	\[
	\psi _j (D) = \dfrac{3}{4} \cdot 2^{-2j} B (1+2^{-2j-2} B )^{-1} (1+2^{-2j} B)^{-1}, 
	\]
	is an non-negative operator on $L^2$, 
	and we can write $\psi_j(D) = \big( \psi_j(D)^{\frac{1}{2}} \big)^{2}$ 
	with the estimate $\| \psi_j(D)^{\frac{1}{2}} \|_{L^2 \to L^2} \leq 1 $, which proves \eqref{1219-2}. 
	
	By Lemma~\ref{lem:1222-4} (1), 
	\[ \| \nabla (1+2^{-2j}B)^{-1} \|_{L^2 \to L^2} 
	   \leq C 	 \| B^{\frac{1}{2}} (1+2^{-2j}B)^{-1} \|_{L^2 \to L^2} , 
	\]
	and we write 
	\[
	B^{\frac{1}{2}} (1+2^{-2j}B)^{-1} 
	=  
	  \bigg[ \Big( \frac{3}{4} (1+2^{(-2j-2)}B)^{-1} (1+2^{-2j}B)^{-1}
	  \Big)^{-\frac{1}{2}} 
	  (1 +2^{-2j}B)^{-1} 
	  \bigg]  
	  2^j \psi_j(D)^{\frac{1}{2}}. 
	\]
	Since the right hand side without $2^j \psi_j(D)^{\frac{1}{2}}$ is bounded uniformly with respect to $j \in \mathbb Z$, we obtain 
	the inequality \eqref{1219-3}. 
\end{pf}

\begin{lem} {\rm(}Proposition 2.3 in \cite{IwKo-preprint}{\rm)}
	A positive constant $C$ exists such that for every $1 \leq p \leq r \leq \infty$ and $j \in \mathbb Z $
	\begin{equation}\label{1219-5}  \| \phi_j(D) \|_{L^r \to L^p} \leq C 2^{3(\frac{1}{r}-\frac{1}{p})j}. 
	\end{equation}
\end{lem}

\begin{lem}\label{lem:1222-5} {\rm(}Embedding in Besov spaces{\rm)}
	Let $s > 0$, $1 \leq r \leq p \leq \infty$ and $1 \leq q \leq \infty$. Then 
	\[
	\begin{split}
	\| u \|_{B^0_{p,q}(A)} \leq &C \| u \|_{B^{3(\frac{1}{r}-\frac{1}{p})}_{r,q}(A)}, 
	\quad u \in B^{3(\frac{1}{r}-\frac{1}{p})}_{r,q}(A), 
\\
\| u \|_{B^0_{p,1}(A)} \leq &C \| u \|_{B^{s}_{p,\infty}(A)}, 
  \quad u \in B^s_{p,\infty} (A). 
\end{split}
\]
Furthermore, 
if $ 1 \leq r \leq p < \infty$ or $1 \leq r < p = \infty$, then 
\[  
	\| \nabla u \|_{L^p} \leq  C \| u \|_{B^{1+3(\frac{1}{r}-\frac{1}{p})}_{r,1}(A)}, 
	\quad u \in B^{1+3(\frac{1}{r}-\frac{1}{p})}_{r,1}(A). 
\]
\end{lem}

\begin{pf}
	The proofs of the first and second inequalities are analogous to those of \cite[Proposition~3.2]{IMT-2019}, based on the $L^r$-$L^p$ boundedness in \eqref{1219-5} and the positivity of the spectrum.

	For the third inequality, we utilize the resolution of the identity:
	\[
	u = \sum_{j \in \mathbb{Z}} \phi_j(D) u,
	\]
	which is justified by an argument similar to that in \cite[Lemma~4.5]{IMT-2019} for the Dirichlet Laplacian. Let $\Phi_j(D) := \phi_{j-1}(D) + \phi_j(D) + \phi_{j+1}(D)$. Then we have the property $\phi_j(D) = \Phi_j(D) \phi_j(D)$. Moreover, 
	if $p < \infty$, then  
	the following $L^r $-$ L^p$ boundedness holds:
	\[
	\| \nabla \Phi_j(D) \|_{L^r \to L^p} \leq C \| B^{\frac{1}{2}} \Phi_j(D) \|_{L^r \to L^p} \leq C 2^{j + 3 \left( \frac{1}{r} - \frac{1}{p} \right) j}.
	\]
		Applying these estimates, we obtain
	\[
	\| \nabla u \|_{L^p} \leq \sum_{j \in \mathbb{Z}} \| \nabla \Phi_j(D) \phi_j(D) u \|_{L^p} \leq C \sum_{j \in \mathbb{Z}} 2^{j + 3 \left( \frac{1}{r} - \frac{1}{p} \right) j} \| \phi_j(D) u \|_{L^p},
	\]
	which proves the proof of the third inequality 
	for $p < \infty$. 
	When $p = \infty$, we utilize the interpolation inequality (see \cite{Niren-1959}) and the elliptic estimate. Let $\max \{ r, n\} < r_0 < \infty$. Then 
	\[
\begin{split}
		\| \nabla \Phi_j(D) \phi_j(D)u \|_{L^\infty}
&	\leq C \| \nabla \Phi_j (D)  \phi_j(D) u \|_{L^{r_0}} ^{1-\frac{3}{r_0}}
	         \| \nabla^2 \Phi_j (D)  \phi_j(D) u \|_{L^{r_0}} ^{ \frac{3}{r_0}}
\\
&		\leq C \| B^{\frac{1}{2}} \Phi_j (D)  \phi_j(D) u \|_{L^{r_0}} ^{1-\frac{3}{r_0}}
	\| B \Phi_j (D) \phi_j(D)  u \|_{L^{r_0}} ^{ \frac{3}{r_0}}         
\\
& \leq C 2^{(1 + \frac{3}{r_0})j} \|  \phi_j(D) u \|_{L^{r_0}}
\\
& \leq C 2^{(1 + \frac{3}{r})j} \|  \phi_j(D) u \|_{L^{r}},
\end{split}
	\]
which completes the proof for $p = \infty$. 	
\end{pf}

\begin{lem}\label{lem:1222-3} {\rm(}Sobolev embedding{\rm)} 
	Let $1 < r \leq p < \infty$. Then 
	\[ \| u \|_{L^p} \leq C \| A^{\frac{3}{2}(\frac{1}{r}-\frac{1}{p})}u \|_{L^r}, 
	\quad \text{for } u \text{ with } \| A^{\frac{3}{2}(\frac{1}{r}-\frac{1}{p})}u \|_{L^r} < \infty .
	\]
	Let $1 < p < \infty  $ and $s\geq 1+ 3(1/r - 1/p)$. Then 
	\[ \| \nabla u \|_{L^p} \leq C \| A^{\frac{s}{2}} u \|_{L^r}, 
		\quad \text{for } u \text{ with } \| A^{\frac{s}{2}}u \|_{L^r} < \infty .
	\]
\end{lem}

\begin{pf}
	The first inequality can be proved in a similar manner to \cite[Theorem~1.4]{FI-2024}. For the second inequality, we apply Lemma~\ref{lem:1222-8} and the first inequality to obtain
	\[
	\| \nabla u \|_{L^p} \leq C \| A^{\frac{1}{2}} u \|_{L^p} \leq C \| A^{\frac{1}{2} + \frac{3}{2} \left( \frac{1}{r} - \frac{1}{p} \right)} u \|_{L^r}.
	\]
	An embedding between regularity indices, ensured by the positivity of the spectrum,  yields
	\[
	\| A^{\frac{1}{2} + \frac{3}{2} \left( \frac{1}{r} - \frac{1}{p} \right)} u \|_{L^r} \leq C \| A^{\frac{s}{2}} u \|_{L^r},
	\]
	which completes the proof.
\end{pf}

\begin{prop}\label{prop:1222-1}
	Let $s > 0$ and $1 \leq r \leq p \leq \infty$. Then for $0< t< 1$
	\[
	\begin{split}
		  \| e^{-tA} u_0 \|_{L^{p}} \leq  
		  &C t^{- \frac{3}{2}(\frac{1}{r}-\frac{1}{p})} \| u_0 \|_{L^{r}}, 
	\quad \text{ for } u \in \mathcal X'_\sigma \cap L^r, 
\\
  \| A^{s}e^{-tA} u_0 \|_{L^{p}} \leq  
  &C t^{- \frac{s}{2}} \| u_0 \|_{L^{p}},
  \quad \text{for } u \in \mathcal X'_\sigma \cap L^r . 
\end{split}
\]
\end{prop}

\noindent 
{\bf Remark.} Proposition~\ref{prop:1222-1} provides local-in-time estimates. While global-in-time estimates with exponential decay could be expected, the present paper focuses primarily on the local-in-time analysis.

\vskip2mm

\begin{pf}
We know that the semigroup $e^{-tA}$ satsifies the Gaussian upper bound (see Lemma~2.2~\cite{IwKo-preprint}), 
and the proof of Proposition~\ref{prop:1222-1} is similar to Theorem~1.6 (ii) in \cite{FI-2024}. 
\end{pf}

\begin{prop} \label{prop:1222-7}
	{\rm(}Maximal regularity{\rm)} 
	Let $T>0$. A positive constant $C$, independent of $T$, exists such that for every $u_0 \in H^1(A)$
	\[  \|  e^{-tA} u_0 \|_{L^2(0,T ; H^2(A))} \leq C \| u_0 \|_{H^1(A)}. 
	\]
For every $f \in L^2 (0,T; H^2 (A))$
	\[  \Big\| \int _0^t e^{-(t-\tau) A} f(\tau) ~d\tau \Big\|_{L^2(0,T ; H^2(A))} \leq C \| f \|_{L^2(0,T; H^2(A))}.
\]
\end{prop}

\begin{pf}The proof is similar to those of Theorems~1.3 and 1.4 in \cite{Iw-2018}.
\end{pf}

	\begin{lem}\label{lem:1225-1}
		Let $j \in \mathbb Z$. For $u \in L^2 _\sigma$ with $\Delta u \in L^2$, we consider the linear operator defined by 
		$u \mapsto \psi_j(D) \Delta u$. Then $\psi _j (D) \Delta $ is extended uniquely to a bounded linear operator on $L^2_\sigma$ and 
		\[
		\sup _{j \in \mathbb Z} 2^{-2j} \| \psi_j(D) \Delta  \|_{L^2 \to L^2} < \infty , \quad 
		\int_{\Omega} \Big( \psi_j(D) (-\Delta u)   \Big) \cdot u ~dx \geq 0 .
        \]
	\end{lem}
	
	\begin{pf}
For every $u \in C_{0,\sigma}^\infty$, the identity $\psi_j(D) (-\Delta) u = \psi_j(D) Au$ holds. We may write
\[ 
\psi_j(D) A = \frac{3}{4} (2^{-2j} A)^2 (1 + 2^{-2j-2} A)^{-1} (1 + 2^{-2j} A)^{-1} \cdot 2^{2j},
\]
which implies the estimate
\[ 
\| \psi_j(D) A u \|_{L^2} \leq C 2^{2j} \| u \|_{L^2}.
\]
This boundedness allows us to extend $\psi_j(D) A$ uniquely to a bounded linear operator on $L^2_\sigma$, since $C_{0,\sigma}^\infty$ is dense in $L^2_\sigma$.

Now, suppose $u \in L^2_\sigma$ and let $\{ u_n \} \subset C_{0,\sigma}^\infty$ be a sequence such that $u_n \to u$ in $L^2$ as $n \to \infty$. For each $n$, it follows from the self-adjointness of the operator $\psi_j(D)$ and the properties of $A$ that
\[ 
\int_{\Omega} \left( \psi_j(D) A u_n \right) \cdot u_n \, dx = \int_{\Omega} | \psi_j(D)^{\frac{1}{2}} A^{\frac{1}{2}} u_n |^2 \, dx \geq 0.
\]
By taking the limit as $n \to \infty$, we obtain
\[ 
\int_{\Omega} \left( \psi_j(D) A u \right) \cdot u \, dx \geq 0,
\]
which completes the proof. 
	\end{pf}
	
	\section{Solutions with $\varepsilon \Delta$ and Regularity} 
	
The goal of this section is to construct solutions to the Navier-Stokes equations with viscosity $\varepsilon \Delta$ for each parameter $\varepsilon > 0$ in the space $C([0, T]; H^1(A))$, supplemented with additional regularity for $\partial_t u$ and $\nabla^4 u$. To this end, we provide two propositions: Proposition~\ref{prop:1128-1} ensures the existence of a local-in-time solution, while Proposition~\ref{prop:1203-1} establishes the higher regularity required to justify the $L^2$ energy estimates with localized spectra, which are essential for the proof of the main theorem.

	\begin{prop}\label{prop:1128-1}
		Let $\varepsilon > 0$ be given. Suppose that the initial data $u_0$ belongs to the $H^1(A)$ Sobolev space associated with the Stokes operator.
		Then there exist a time $T>0$ and a unique solution $u \in C([0,T] , H^1 (A))$ to the following integral equation:
		\begin{equation}\label{eq:IE}
			u(t) = e^{-t \varepsilon A} u_0 - \int _0^t e^{-(t-\tau) \varepsilon A} \mathbb P \Big( (u(\tau) \cdot \nabla) u(\tau) \Big) d\tau , 
			\quad t \in [0,T].
		\end{equation}
	\end{prop}
	\begin{pf}
		We begin by establishing the existence of a local solution $u$. 
		We first choose a constant $M_0$ satisfying
		\[
		\| u_0 \|_{ H^1 (A)} \leq M_0. 
		\]
		Let $T > 0$ be a time to be specified later.
		We apply the $\text{Banach fixed point theorem}$ to the complete metric space $X_T$, which is defined as
		\[
		X_T := \left\{ u \in C([0,T], H^1 (A))\, \middle| \, \| u \|_{X_T} \leq 2M_0 \right\},
		\]
		equipped with the norm
		\[
		\| u \|_{X_T} := \sup _{ t \in [0,T]} \| u(t) \|_{H^1 (A)}, 
		\]
		and the metric
		\[
		d(u,v) := \| u-v \|_{X_T}.
		\]
		
		It is straightforward to see that 
		\[
		\| e^{-t\varepsilon A} u_0 \|_{X_T} \leq \| u_0 \|_{H^1 (A)} \leq M_0. 
		\]
		We apply the smoothing property and the $L^{3/2} - L^2$ estimate  in Proposition~\ref{prop:1222-1} 
		for the semigroup $e^{-(t-\tau)\varepsilon A}$, 
		and subsequently use the boundedness of the projection $\mathbb P$ 
		on $L^{3/2}$ from Lemma~\ref{lem:1222-2}  and the H\"older inequality. 
		The following estimate holds:
		\begin{equation}\label{1128-4}
			\begin{split}
				\Big\| \int_0^t e^{-(t-\tau )\varepsilon A} \mathbb P \big( (u \cdot \nabla ) u \big) d\tau
				\Big\| _{H^1(A)}
				\leq 
				& 
				C \int _0^t \Big( \varepsilon(t-\tau) \Big) ^{-\frac{1}{2}-\frac{3}{2}(\frac{2}{3}-\frac{1}{2})} 
				\| \mathbb P \big( (u \cdot \nabla ) u \big) \|_{L^\frac{3}{2}} d\tau 
				\\
				\leq & 
				C \varepsilon^{-\frac{3}{4} } \int _0^t (t-\tau) ^{-\frac{3}{4} } \| u \|_{L^6} \| \nabla u \|_{L^2} d\tau .
			\end{split}
		\end{equation}
		The Sobolev embedding $H^1 (A) \hookrightarrow L^6$ from Lemma~\ref{lem:1222-3} and the norm equivalence 
		$\| \nabla u \|_{L^2} \simeq \| u \|_{H^1(A)}$ in Lemma~\ref{lem:1222-4} yield the following bound:
		\begin{equation}\label{1128-5}
			\begin{split}
				\Big\| \int_0^t e^{-(t-\tau )\varepsilon A} \mathbb P \big( (u \cdot \nabla ) u \big) d\tau
				\Big\| _{H^1(A)}
				\leq & 
				C \varepsilon^{-\frac{3}{4} } T^{\frac{1}{4}} \| u \|_{X_T} ^2 
				\leq C \varepsilon^{-\frac{3}{4} } T^{\frac{1}{4}} M_0^2 ,
			\end{split}
		\end{equation}
		for any $u \in X_T$. Analogously, we obtain the difference estimate:
		\begin{equation}\label{1128-1}
			\begin{split}
				&	\Big\| \int_0^t e^{-(t-\tau )\varepsilon A} \mathbb P \big( (u \cdot \nabla ) u \big) d\tau
				- \int_0^t e^{-(t-\tau )\varepsilon A} \mathbb P \big( (v \cdot \nabla ) v \big) d\tau
				\Big\| _{\dot H^1(A)}
				\\
				\leq & 
				C \varepsilon^{-\frac{3}{4} } T^{\frac{1}{4}} 
				(\| u \|_{X_T} + \| v \|_{X_T}) \| u-v \|_{X_T}
				\\
				\leq & C \varepsilon^{-\frac{3}{4} } T^{\frac{1}{4}} M_0 \| u-v \|_{X_T},
			\end{split}
		\end{equation}
		for any $u, v \in X_T$. 
		By these inequalities and by choosing $T$ such that $C \varepsilon^{-\frac{3}{4}} T^{\frac{1}{4}} \ll 1$, 
		we ensure that the map 
		\[
		X_T \ni u \quad \mapsto \quad 
		e^{-t \varepsilon A} u_0 - \int _0 ^t e^{-(t-\tau) \varepsilon A} \mathbb P \Big( (u \cdot \nabla ) u   \Big) ~d\tau 
		\in X_T
		\]
		is well-defined and contractive. We then find a unique fixed point $u$ in $X_T$, which is the desired solution.

		We now prove the uniqueness of the solution in the space $C([0,T] ,H^1 (A))$. 
		Suppose that $u$ and $v$ are two solutions of the integral equation 
		with the same initial data $u_0$ in $ H^1(A)$. We define the quantity 
		\[ 
		M := \sup _{t \in [0,T]} (\| u \|_{X_T} + \| v \|_{X_T}). 
		\] 
		Let $0 < t_0 < T$ and consider the difference $u-v$ on the time interval $[0,t_0]$. 
		A similar argument to that used to derive inequality \eqref{1128-1} implies that 
		\[
		\| u - v\|_{X_{t_0}} 
		\leq C \varepsilon ^{-\frac{3}{4}} t_0 ^{\frac{1}{4}} M \| u-v \|_{X_{t_0}}.
		\]
		By choosing $t_0$ sufficiently small such that $C \varepsilon ^{-\frac{3}{4}}t_0 ^{\frac{1}{4}} < 1$, the above inequality yields 
		$\| u - v\|_{X_{t_0}} = 0$, which means $u = v$ on the time interval $[0,t_0]$. 
		Due to the time continuity, we have $u(t_0) = v(t_0)$ in $H^1 (A)$. 
		Treating $u(t_0)$ as the new initial data, we can apply the same argument iteratively on the time interval $[t_0 , 2t_0]$, provided $2t_0 \leq T$. 
		By repeating this procedure finitely many times, the uniqueness is established throughout the entire interval $[0,T]$.
	\end{pf}

\vskip3mm 
	
	We will need the regularity of the solutions to justify the energy estimate in the proof of our theorem. 
	An analogous argument to Giga and Miyakawa~\cite{GiMi-1985} combined with maximal regularity ensures the following regularity.

	\begin{prop}\label{prop:1203-1}
		Let $u \in C([0,T], H^1 (A))$ be the unique local solution to the integral equation \eqref{eq:IE} constructed in the previous proposition. 
		Then, 
		\[ u \in L^2 (0,T ; H^3 (\Omega)) \cap L^2 (0,T; H^2(A)), 
		\]
		and for any $t_0$ such that $0 < t_0 < T$, the solution $u$ satisfies the following higher regularity estimates:
		\[
		u \in L^2(t_0,T; H^4(\Omega)) \quad \text{and} \quad 
		\partial _t u \in L^2(t_0,T; H^2(A)).
		\]
		Here, $H^2 (A)$ denotes the Sobolev space associated with the Neumann Stokes operator $A$, and $H^4 (\Omega)$ is the standard Sobolev space equipped with the domain $\Omega$ (without boundary conditions).
	\end{prop}
	
	\begin{pf}
		Step 1. 
		We first prove that the solution $u$ belongs to $L^\infty _{\rm loc} ((0,T] , H^s (A))$ for any $s < 2$. 
		If $1 < s_1 < 3/2$, a similar argument to those used in deriving \eqref{1128-4} and \eqref{1128-5} shows that 
		\[
		\begin{split}
			\| u(t) \|_{H^{s_1}(A)} 
			\leq 
			&Ct^{-\frac{s_1-1}{2}} \| u_0 \|_{H^1 (A)} 
			+ C \int _0^t \Big( \varepsilon (t-\tau)\Big)  ^{-\frac{s_1}{2} - \frac{3}{2}(\frac{2}{3}-\frac{1}{2})} 
			\| \mathbb P (u \cdot \nabla )u \|_{L^{\frac{3}{2}}} ~ d\tau 
			\\
			\leq 
			& Ct^{-\frac{s_1-1}{2}} \| u_0 \|_{H^1 (A)} 
			+ C \varepsilon ^{-\frac{s_1}{2}- \frac{1}{4}} t ^{-\frac{s_1}{2}+ \frac{3}{4}} \sup_{\tau \in [0,t]} \| u(\tau) \|_{H^1(A)}^2 ,
		\end{split}
		\]
		which implies that $u \in L^\infty _{\rm loc} ((0,T] , H^{s_1} (A))$. 
		
		Next, we choose indices $s_2, p_1 , p_2$ satisfying the conditions
		\[ \frac{3}{2} \leq s_2 < 2, , \quad 6 < p_1 < \infty, \quad 2 < p_2 < 3, \quad \frac{1}{2} = \frac{1}{p_1} + \frac{1}{p_2}, 
		\]and it holds that 
		\[
		\begin{split}
			\| u(t) \|_{ H^{s_2}(A)} 
			\leq 
			&Ct^{-\frac{s_2-1}{2}} \| u_0 \|_{ H^1 (A)} 
			+ C \int _0^t \Big( \varepsilon (t-\tau)\Big)  ^{-\frac{s_2}{2} }
			\| \mathbb P (u \cdot \nabla )u \|_{L^{2}} ~ d\tau 
			\\
			\leq 
			&Ct^{-\frac{s_2-1}{2}} \| u_0 \|_{ H^1 (A)} 
			+ C \varepsilon ^{-\frac{s_2}{2} }  t^{1-\frac{s_2}{2} }
			\sup _{\tau \in [0,t]} \| u \|_{L^{p_1}} \| \nabla u \|_{L^{p_2}}.
		\end{split}
		\]
		Now, we select the regularity index $s_1$ such that 
		\[
		\max \bigg\{ 3 \Big( \frac{1}{2} - \frac{1}{p_1} \Big) , 3 \Big( \frac{1}{2} - \frac{1}{p_2} \Big) + 1 \bigg\} 
		< s_1   < \frac{3}{2}. 
		\]
		The Sobolev embedding  $H^{3(\frac{1}{2}-\frac{1}{p_1})}(A) \hookrightarrow L^{p_1}$ from Lemma~\ref{lem:1222-3} combined with the inclusion $H^{s_1}(A) \hookrightarrow H^{3(\frac{1}{2}-\frac{1}{p_1})}(A)$ (since $3(1/2 - 1/p_1) < s_1$) yields
		\[
		\| u \|_{L^{p_1}} \leq C \| u \|_{H^{s_1}(A)}. 
		\]
		For the second term, by using the $\boldsymbol{\text{Littlewood-Paley decomposition}}$ (or $\boldsymbol{\text{resolution of identity}}$) and the $\boldsymbol{\text{Besov space embedding}}$ $H^{s_1}(A) = B^{s_1}_{2,2} (A) \hookrightarrow B^{1+ 3 (\frac{1}{2}- \frac{1}{p_2})}_{p_2, \infty}(A)$ from Lemma~ \ref{lem:1222-5}, followed by the triangle inequality, we obtain
		\[
		\begin{split}
			\| \nabla u \|_{L^{p_2}} 
			\leq
			& \sum _{j \in \mathbb Z} 2^j \| \phi_j(D) u \|_{L^{p_2}}
			\leq C \sum _{j \in \mathbb Z} 2^{j + 3(\frac{1}{2} - \frac{1}{p_2})j} \| \phi_j(D) u \|_{L^{2}}
			= C \| u \|_{B^{1+3(\frac{1}{2}-\frac{1}{p_2})}_{2,1}(A)}
			\\
			\leq & C \| u \|_{H^{s_1} (A)}. 
		\end{split}
		\]
		Therefore, we conclude that if $3/2 < s_2 < 2$, then
		\[
		\begin{split}
			\| u(t) \|_{ H^{s_2}(A)} 
			\leq 
			&Ct^{-\frac{s_2-1}{2}} \| u_0 \|_{ H^1 (A)} 
			+ C \varepsilon ^{-\frac{s_2}{2} }  t^{1-\frac{s_2}{2} }
			\sup _{\tau \in [0,t]} \| u \|_{ H^{s_1}(A)}^2 < \infty . 
		\end{split}
		\]

		\vskip2mm 
		
		\noindent 
		Step 2. 
		The goal of Step 2 is to prove that the time derivative of the fractional power of the solution, $\partial_t A^{1/2} u(t)$, belongs to $L^2 (0,T; L^2)$.
		We first prove that $\partial _t u \in L^2 (0,T; L^2)$. To this end, we start with maximal regularity estimate Proposition~\ref{prop:1222-7}, which gives
		\[
		\begin{split}
			\| u \|_{L^2 (0,T ; H^2 (A))} 
			\leq
			& C \| u_0 \|_{ H^1 (A)} 
			+ C \Big\| \mathbb P \Big( (u \cdot \nabla ) u \Big)  \Big\|_{L^2 (0,T ; L^2)}
			\\
			\leq 
			& C \| u_0 \|_{ H^1 (A)} 
			+ C T^{\frac{1}{2}}  \sup_{t\in[0,T]} \| u(t)  \|_{ H^{s_1} (A)} ^2 , \quad 
			\text{with } 1 < s_1 < \frac{3}{2}. 
		\end{split}
		\]
		Since $u \in C([0,T], H^{s_1}(A))$ with $s_1 > 1$, the term $\sup_{t\in[0,T]} \| u(t) \|_{ H^{s_1} (A)}$ is finite. 
		Consequently, this bound, combined with the maximal regularity estimate, implies that 
		\[
		\partial_t u \in L^2 (0,T ; L^2). 
		\]
		
		We now show the stronger regularity $\partial _t A^{\frac{1}{2}} u(t) \in L^2 (0,T; L^2)$. 
		We rewrite $A^{1/2}$ on the integral equation using the identity $A^{1/2} = A B^{-\frac{1}{2}}$ (see Lemma~\ref{lem:1222-8}):
		\[
		\begin{split}
			A^{\frac{1}{2}} u(t)
			=& A^{\frac{1}{2}} e^{-t \varepsilon A} u_0 
			- \int _0^t A e^{-(t-\tau) \varepsilon A} \mathbb P B^{-\frac{1}{2}} \Big( (u \cdot \nabla) u\Big) d\tau
			\\
			=& A^{\frac{1}{2}} e^{-t \varepsilon A} u_0 
			- \varepsilon ^{-1} \int _0^t  ( \partial _\tau e^{-(t-\tau) \varepsilon A} ) \mathbb P B^{-\frac{1}{2}} \Big( (u \cdot \nabla) u\Big) d\tau.
		\end{split}
		\]
		Applying integration by parts (with respect to $\tau$ on the integral term), we obtain
		\[
		\begin{split}
			A^{\frac{1}{2}} u(t)
			=& A^{\frac{1}{2}} e^{-t \varepsilon A} u_0 
			- \varepsilon ^{-1} \left[ e^{-(t-\tau) \varepsilon A} \mathbb P B^{-\frac{1}{2} } \Big( (u(\tau) \cdot \nabla) u(\tau) \Big) \right]_{\tau=0}^{\tau=t}
			\\
			& \qquad + \varepsilon ^{-1} \int _0^t e^{-(t-\tau) \varepsilon A} \mathbb P B^{-\frac{1}{2} } \partial _\tau \Big( (u \cdot \nabla) u \Big) d\tau
			\\
			=& A^{\frac{1}{2}} e^{-t \varepsilon A} u_0 
			- \varepsilon ^{-1} \mathbb P B^{-\frac{1}{2} } \Big(  (u(t) \cdot \nabla) u(t) \Big) 
			+ \varepsilon ^{-1}e^{-t \varepsilon A} \mathbb P  B^{-\frac{1}{2}} \Big( (u_0 \cdot \nabla) u_0\Big) 
			\\
			& \qquad 
			+ \varepsilon ^{-1} \int _0^t   e^{-(t-\tau) \varepsilon A}  \mathbb P B^{-\frac{1}{2}} \nabla  \partial _\tau (u \otimes u )  d\tau. 
		\end{split}
		\]
		We now examine the time derivative $\partial _t A^{\frac{1}{2}} u(t)$ for each term.
		It is straightforward to see that the derivatives belong to $L^2 (0,T ; L^2)$ for the first and third terms (semigroup terms), since $u_0$ is sufficiently regular. 
		We directly consider the derivative of the second term:
		\[
		\partial _t \left[ \varepsilon ^{-1} \mathbb P B^{-\frac{1}{2} } \Big( (u(t) \cdot \nabla) u(t) \Big) \right] = \varepsilon^{-1} \mathbb P B^{-\frac{1}{2}} \nabla \cdot \partial_t (u \otimes u)(t).
		\]
		Applying the boundedness of $B^{-\frac{1}{2}} \nabla$ on $L^2$ and the product rule, we have that 
		\begin{equation}\label{1201-1}
			\begin{split}
				\| \partial _t \varepsilon ^{-1} \mathbb P B^{-\frac{1}{2}} \big( (u(t) \cdot \nabla ) u(t) \big) \| _{L^2 (0,T ; L^2)}
				&\leq C \varepsilon ^{-1} \| B^{-\frac{1}{2}} \nabla \|_{L^2 \to L^2} \| \partial _t \big( u \otimes u \big) \| _{L^2 (0,T ; L^2)}
				\\
				&\leq C \varepsilon ^{-1} \| u \|_{L^\infty (0,T ; L^\infty)} \| \partial _t u \|_{L^2 (0,T ; L^2)} < \infty .
			\end{split}
		\end{equation}
		The derivative of the fourth term is given by applying the Leibniz rule for integrals: 
		\[
		\begin{split}
			&
			\partial _t \bigg[ \varepsilon ^{-1} \int _0^t   e^{-(t-\tau) \varepsilon A}  \mathbb P  B^{-\frac{1}{2}}\nabla  \partial _\tau (u \otimes u ) ~d\tau\bigg]
			\\
			= 
			&\varepsilon ^{-1} \mathbb P B^{-\frac{1}{2}}\nabla \partial _t ( u(t) \otimes u(t))
			-  \int_0^t A e^{-(t-\tau) \varepsilon A} \mathbb P B^{-\frac{1}{2}} \nabla \partial _\tau (u \otimes u) ~d\tau .
		\end{split}
		\]
		The first term on the right-hand side is handled in the same manner as in \eqref{1201-1}. 
		The second term on the right-hand side is bounded by maximal regularity estimate Proposition~\ref{prop:1222-7} and the previous estimate \eqref{1201-1} as follows:
		\[
		\begin{split}
			\Big\|  \varepsilon ^{-1} \int_0^t A e^{-(t-\tau) \varepsilon A} \mathbb P B^{-\frac{1}{2}}  \nabla \partial _\tau(u \otimes u) ~d\tau
			\Big\| _{L^2 (0,T ; L^2)}
			\leq & C(\varepsilon) \|  \mathbb PB^{-\frac{1}{2}}  \nabla \partial _\tau (u \otimes u)\|_{L^2 (0,T ; L^2)} 
			\\
			\leq & C (\varepsilon) \| u \|_{L^\infty (0,T ; L^\infty)} \| \partial _t u  \|_{L^2 (0,T ; L^2)} 
			\\
			< & \infty ,
		\end{split}
		\]
		where $C(\varepsilon)$ is a positive constant depending on $\varepsilon$. 
		Consequently, we conclude that $\partial _t A^{\frac{1}{2}} u \in L^2 (0,T ; L^2) $.

		\vskip2mm

		\noindent Step 3. 
		The goal of Step 3 is to establish the spatial regularity of the solution, specifically proving that $u$ belongs to $L^2(0,T; H^3(\Omega))$.
		
		We start from the governing equation, the $\boldsymbol{\text{Neumann Stokes equation}}$, written in terms of the operator $A$:
		\[
		\partial _t u + \varepsilon A u + \mathbb P \big( (u \cdot \nabla) u\big) = 0.
		\]
		We isolate the solution $u$ by applying the inverse operator $A^{-1}$:
		\begin{equation}\label{1203-1}
			u = - A^{-1} \Big( \varepsilon^{-1} \mathbb P \big( (u \cdot \nabla) u \big) + \varepsilon^{-1} \partial _t u \Big).
		\end{equation}
		We then apply the $\boldsymbol{\text{elliptic regularity estimate}}$ (see Lemma~4.4 in Kozono-Yanagisawa~\cite{KoYa-2009}) for the $\boldsymbol{\text{Neumann Stokes operator}}$ to obtain the spatial estimate:
		\[
		\| u\|_{H^3 (\Omega)} 
		\leq C \Big\| \varepsilon^{-1} \mathbb P ((u \cdot \nabla) u) \Big\|_{H^1 (\Omega)} + C \Big\| \varepsilon^{-1} \partial _t u \Big\|_{H^1 (\Omega)}.
		\]
		Here, the Sobolev space $H^k(\Omega)$ is the $\boldsymbol{\text{standard Sobolev space}}$ on the domain $\Omega$, defined by a subspace of $L^2(\Omega)$ where the weak derivatives up to order $k$ belong to $L^2(\Omega)$.
		We utilize the boundedness of the projection $\mathbb P$ on $H^1 (\Omega)$ in Lemma~\ref{lem:1222-9} and the norm equivalence between $H^1(\Omega)$ and $H^1(A)$ to estimate the nonlinear term and the time derivative term, respectively. 
		From Step 2, we note that $A^{1/2} \partial _t u \in L^2 (0,T ; L^2(\Omega))$.
		By taking the $L^2$ norm with respect to the time variable, we obtain the following final estimate:
\[ \begin{split} \| u\|_{L^2 (0,T;H^3 (\Omega))} \leq &C \| \mathbb P ((u \cdot \nabla) u) \|_{L^2 (0,T; H^1 (\Omega))} +C \| A^{1/2} \partial _t u \|_{L^2 (0,T; L^2)} \\ \leq &C T^{\frac{1}{2}} \| u \|_{L^\infty (0,T ; H^s(\Omega))} \| u \|_{L^2(0,T ; H^2(\Omega))} +C \| A^{\frac{1}{2}} \partial _t u \|_{L^2 (0,T; L^2)} <\infty , \end{split} \] where $3/2 < s < 2$.
		The finiteness follows from the regularity results obtained in Step 1 and Step 2.

		\vskip2mm

		\noindent Step 4. 
		We claim $u \in L^2 (0,T ; H^4 (\Omega))$.
		The governing equation is
		\[
		\partial _t u + \varepsilon Au + \mathbb P \big( (u \cdot \nabla) u \big) = 0 \quad \text{ in } L^2 ,
		\text{ for almost every } t \in (0,T) .
		\]
		We apply the spectral restriction operator $\phi_j (D)$ to the equation:
		\[
		\partial _t \phi_j(D) u + \varepsilon A \phi_j(D) u + \phi_j(D) \mathbb P \big( (u \cdot \nabla ) u \big) = 0 \quad
		\text{ in } L^2 , \quad \text{ for almost every } t \in (0,T).
		\]
		
		We first justify the time differentiability of the above equation.
		The second term, $\varepsilon A \phi_j(D) u$, is time differentiable since $A \phi_j(D)$ is a bounded operator on $L^2$ and $\partial _t u \in L^2 (0,T ; L^2)$. Thus $\partial _t (\varepsilon A \phi_j (D) u) \in L^2 (0,T ; L^2)$.
		The regularity $\partial _t A^{\frac{1}{2}} u \in L^2 (0,T ; L^2)$ from Step 2 implies that the time derivative of the nonlinear term is justified:
		\[
		\partial _t \left[ \phi_j (D) \mathbb P \big( (u\cdot \nabla) u \big) \right]
		= \phi_j (D) \mathbb P \big( (\partial _t u\cdot \nabla) u \big) + \phi_j (D) \mathbb P \big( ( u\cdot \nabla) \partial _t u \big).
		\]
		Since all terms on the right-hand side of the restricted equation are time differentiable in $L^2(0,T; L^2)$, we conclude that $\partial _t^2 \phi_j(D) u$ is justified as an element of $L^2 (0,T; L^2)$.
		The twice-differentiated equation is:
		\[
		\partial _t^2 \phi_j(D)u + \varepsilon A \partial _t \phi_j(D) u + \phi_j (D) \mathbb P \big( (\partial _t u\cdot \nabla) u \big) + \phi_j (D) \mathbb P \big( ( u\cdot \nabla) \partial _t u \big)
		= 0.
		\]
		
		This allows us to write $\partial _t \phi_j(D) u$ using the Duhamel formula. For almost every $t_0,t$ with $0 < t_0 < t <T$, we have
		\[
		\begin{split}
			\phi_j(D) \partial _t u(t)
			=& \phi_j(D) e^{-(t-t_0) \varepsilon A} \partial _t u(t_0)
			\\
			& - \phi_j(D) \int_{t_0}^t e^{-(t-\tau) \varepsilon A} \mathbb P
			\big( (\partial _\tau u \cdot \nabla ) u + (u \cdot \nabla) \partial _\tau u\big) ~d\tau
			\quad \text{ in } L^2 .
		\end{split}
		\]
		Since the equality holds for all $j \in \mathbb Z$ and each term without $\phi_j(D)$ is in $L^2$, we may remove $\phi_j(D)$ to obtain an equality for $\partial_t u$.
		\[
		\begin{split}
			\partial _t u(t)
			=& e^{-(t-t_0) \varepsilon A} \partial _t u(t_0)
			- \int_{t_0}^t e^{-(t-\tau) \varepsilon A} \mathbb P
			\big( (\partial _\tau u \cdot \nabla ) u + (u \cdot \nabla) \partial _\tau u\big) ~d\tau
			\quad \text{ in } L^2 .
		\end{split}
		\]
		Furthermore, this implies that $\partial _t u$ is continuous with respect to $t$ in the $L^2$ topology.
		
		Next, we estimate $\|A \partial_t u\|_{L^2(t_0, T; L^2)}$.
		By applying the operator $A$ and taking the $L^2 (t_0,T ; L^2))$ norm on the equation for $\partial_t u$, we obtain the estimate:
		\[
		\begin{split}
			&	\| A \partial _t u \|_{L^2 (t_0 ,T; L^2)}
			\\
			\leq
			&C \| \partial _t u(t_0) \|_{H^1(A)} + C \left\| \int_{t_0}^t A e^{-(t-\tau) \varepsilon A} \mathbb P \big( (\partial _\tau u \cdot \nabla ) u + (u \cdot \nabla) \partial _\tau u\big) ~d\tau \right\|_{L^2(t_0, T; L^2)}
			\\
			\leq
			& C \| \partial _t u(t_0) \|_{H^1(A)} + C \| (\partial _\tau u \cdot \nabla) u + (u \cdot \nabla ) \partial _\tau u \|_{L^2 (t_0,T ; L^2)}
			\\
			\leq
			& C \| \partial _t u(t_0) \|_{H^1(A)} + C \| \partial _\tau u \| _{L^2 (t_0,T ; L^6)} \| \nabla u \|_{L^\infty (t_0,T ; L^3)}
			+ C \| u \|_{L^\infty (t_0,T ; L^\infty)} \| \nabla \partial_\tau u \|_{L^2 (t_0,T; L^2)},
		\end{split}
		\]
		where we have used maximal regularity estimate Proposition~\ref{prop:1222-7}. 
		We choose $t_0 \in (0, T)$ such that $A^{\frac{1}{2}}\partial _t u(t_0) \in L^2$, which is possible since $\partial _t u \in L^2 (0,T; H^1(A))$ by Step 2, 
		and the first term $\| \partial _t u(t_0) \|_{H^1(A)}$ is finite.
		For the second term, by Sobolev embedding $H^1(\Omega) \hookrightarrow L^6$ and $A^{\frac{1}{2}} \partial _t u \in L^2 (0,T; L^2)$ from Step 2, we have
		\[
		\| \partial _t u \|_{L^2(t_0,T ; L^6)}
		\leq C \| \partial _t u \|_{L^2 (t_0 ,T ; H^1 (\Omega))}
		\leq C \| \partial _t u \|_{L^2 (t_0 ,T ; H^1 (A)) } < \infty .
		\]
		We also estimate $\| \nabla u \|_{L^3}$ using the embedding theorem Lemma~\ref{lem:1222-3}: 
		\[
		\| \nabla u \|_{L^3} 
		\leq C \| u \|_{H^{s}(A)} , \quad \frac{3}{2} < s < 2,
		\]
		which is finite by Step 1.
		For the fourth term, the estimate is:
		\[
		\| u \|_{L^\infty (t_0, T; L^\infty)} \| \nabla \partial _\tau u \|_{L^{2} (t_0,T ; L^2)}
		\leq C \| u \|_{L^\infty (t_0, T ; H^s(A))} \| A^{\frac{1}{2}} \partial _\tau u \|_{L^2 (t_0,T ; L^2)}
		< \infty .
		\]
		Therefore, we conclude that $A \partial _t u \in L^2 (t_0 , T ; L^2)$.

		Finally, we estimate  $\|u\|_{L^2(t_0, T; H^4(\Omega))}$. 
		It follows from the elliptic estimate in Lemma~\ref{lem:1222-10} and the equations that 
		\[
		\|u\|_{H^4(\Omega)} = \| A^{-1} Au \|_{H^4(\Omega)} 
		\leq C \|Au\|_{H^2(\Omega)} \leq C \| \mathbb P((u \cdot \nabla)u) \|_{H^2(\Omega)} + C \|\partial_t u\|_{H^2(\Omega)}. 
		\]
		The $H^4(\Omega)$ norm of the solution is then estimated as:
		\[
		\begin{split}
			\| u \|_{H^4 (\Omega)} 
			\leq 
			& C (\| u \|_{L^6} + \| \nabla ^2 u \|_{L^6}) \| \nabla u  \|_{L^3} +C \| u \|_{L^\infty} \| u \|_{H^3 (\Omega)} + C\| A \partial _t u \|_{L^2} 
			\\
			\leq 
			& C \| u \|_{H^3(\Omega)} \| u \|_{H^s(A)}+ C\| A \partial _t u \|_{L^2}, 
			\quad \text{ with }\frac{3}{2} < s < 2. 
		\end{split}
		\]
		By taking the $L^2$ norm in $t$, we obtain:
		\[
		\| u \|_{L^2 (t_0,T ; H^4 (\Omega))}
		\leq C \| u \|_{L^\infty (t_0 ,T ; H^s(A))} \| u \|_{L^2 (t_0,T ; H^3 (\Omega))}
		+ C \| A \partial _t u \|_{L^2 (t_0,T ; L^2)}< \infty.
		\]
		The finiteness follows from the estimates in Step 1, Step 3, and the conclusion $A \partial _t u \in L^2 (t_0 , T ; L^2)$ just established.
	\end{pf}

	\section{Uniform Boundedness with respect to $\varepsilon$}

The goal of this section is to establish the uniform boundedness of the solution $u_\varepsilon$ in $L^\infty(0, T; B^{\frac{5}{2}}_{2,1}(A))$ with respect to $\varepsilon$. To this end, we provide two propositions: Proposition~\ref{prop:1225-1} ensures the regularity in $L^\infty(0, T; B^{\frac{5}{2}}_{2,1}(A))$, while Proposition~\ref{prop:1218-1} establishes the uniform boundedness and the time continuity.

\begin{prop} \label{prop:1225-1} 
	Let $0 < \varepsilon < 1$ and $u_0 \in B^{\frac{5}{2}}_{2,1}(A)$. Then, there exists a time $T > 0$ such that the integral equation \eqref{eq:IE} admits a unique solution $u_\varepsilon$ satisfying
	\[
	u_\varepsilon \in C([0,T], H^1(A)) \cap L^\infty(0,T; B^{\frac{5}{2}}_{2,1}(A)).
	\]
\end{prop}

\begin{pf}
		Step 1. 
It follows from Propositions~\ref{prop:1128-1}, \ref{prop:1203-1} that a unique solution $u_\varepsilon \in C([0,T], H^1(A))$ exists and possesses the following regularity properties:
\[
\begin{split}
	& u_\varepsilon \in 
L^2 ((0, T], H^2(A))
\cap L^2 ((0, T], H^3(\Omega))
\cap L^2 _{\rm loc }((0, T], H^4(\Omega)), \quad
\\
& 
\partial _t u_\varepsilon \in L^2 _{\rm loc} ((0,T], H^2(A)).
\end{split}
\]
Applying the Laplacian $\Delta$ to the partial differential equations yields the following differentiated equations: 
\[
\partial _t \Delta u_\varepsilon - \varepsilon \Delta^2 u_\varepsilon + (u_\varepsilon \cdot \nabla ) \Delta u_\varepsilon
+ N(u_\varepsilon,u_\varepsilon) = 0
\quad \text{ in } L^2 (0,T; L^2),
\]
and  \eqref{1225-1} in the $L^2$ setting  yields that 
the nonlinear remainder term $N(u,v)$ is given by 
\[
N(u,v) := (\Delta u \cdot \nabla ) v + 2 (\nabla u \cdot \nabla) \nabla v
+ \nabla {\rm div}\, \big((u \cdot \nabla) v\big) .
\]
We note that the term $2 (\nabla u \cdot \nabla) \nabla v$ is understood componentwise. Specifically, for the $k$-th component ($k= 1,2,3$), it is defined by:
\[
\Big( 2 (\nabla u \cdot \nabla) \nabla v \Big) _k
= \sum _{i,j = 1}^3
2 (\partial _{x_i} u_j) ( \partial_{x_j} \partial_{x_i} v_k).
\]
We now apply the spectral restriction operator $\psi _j (D)$ to the differentiated equation to obtain: 
\begin{equation}\label{0910-1}
	\psi_j(D)
	\Big(
	\partial _t \Delta u_\varepsilon - \varepsilon \Delta^2 u_\varepsilon + (u_\varepsilon \cdot \nabla ) \Delta u_\varepsilon
	+ N(u_\varepsilon,u_\varepsilon)
	\Big) = 0
	\quad \text{ in } L^2 (0,T; L^2).
\end{equation}
In the subsequent steps, we will multiply the equation \eqref{0910-1} by an appropriate test function, such as $\Delta u_\varepsilon$, and integrate in space to obtain energy estimates. We structure the proof as follows: we analyze the terms involving $\partial _t \Delta u_\varepsilon$, $\varepsilon \Delta^2 u_\varepsilon$, and $(u _\varepsilon \cdot \nabla ) \Delta u_\varepsilon$ in Step~2; we estimate the nonlinear remainder term $N(u_\varepsilon, u _\varepsilon)$ in Step 3; and finally we prove the existence of $T$ such that $u \in L^\infty (0,T; B^{\frac{5}{2}}_{2,1}(A))$  in Step 4.

	\vskip2mm 

\noindent 
Step 2. 
We analyze the terms in the equation \eqref{0910-1} multiplied by the test function $\Delta u_{\varepsilon}$.
The time derivative term is converted into a total time derivative:
\[
\begin{split}
	\int_{\Omega} \psi_j (D) \partial_t \Delta u_\varepsilon \cdot \Delta u_\varepsilon \,dx
	&= \int_{\Omega} \partial_t \left( \psi_j (D)^{\frac{1}{2}}\Delta u_\varepsilon \right) \cdot \left( \psi_j (D)^{\frac{1}{2}} \Delta u_\varepsilon \right) \,dx
	\\
	&= \dfrac{1}{2} \partial _t \int _{\Omega} |\psi_j(D)^{\frac{1}{2}}\Delta u_\varepsilon |^2 \,dx
	= \dfrac{1}{2} \partial _t \| \psi_j (D)^{\frac{1}{2}} \Delta u_\varepsilon \|_{L^2}^2.
\end{split}
\]

For the viscous term, applying Lemma~\ref{lem:1225-1}, we ensure a favorable sign:
\[
- \int_{\Omega} \Big( \psi_j(D) \big( \varepsilon \Delta^2 u_\varepsilon \big) \Big) \cdot \Delta u_{\varepsilon} \,dx = \varepsilon \int_{\Omega} \Big( \psi_j(D)^{\frac{1}{2}} \Delta u_{\varepsilon} \Big) \cdot \Big( \psi_j(D)^{\frac{1}{2}} \Delta^2 u_{\varepsilon} \Big) \,dx \geq 0 .
\]

We estimate the convection term. Using the symmetry of the operator $\psi_j (D)$, we first shift the operator:
\[
\int_{\Omega} \Big( \psi _j (D)\big((u_\varepsilon \cdot \nabla) \Delta u_\varepsilon\big)\Big) \cdot \Delta u_\varepsilon \,dx
= \int_{\Omega}\big((u_\varepsilon \cdot \nabla) \Delta u_\varepsilon\big) \cdot \psi _j (D)\Delta u_\varepsilon \,dx. 
\]
The next step involves an identity based on integration by parts together with the boundary condition $u_\varepsilon \cdot \nu = 0$: 
\[
\int_{\Omega}\big((u_\varepsilon \cdot \nabla) \Delta u_\varepsilon\big) \cdot \psi _j (D)\Delta u_\varepsilon \,dx = -\int_{\Omega} \big((u_\varepsilon \cdot \nabla)\psi _j (D) \Delta u_\varepsilon\big) \cdot \Delta u_\varepsilon \,dx. 
\]
Then, using the specific definition of the operator $\psi_j(D)$ via the spectral operator $B$:
\[
\psi_j(D) = \sum _{l=0}^1 (-1)^{l+1} (1+ 2^{-2j-2l} B)^{-1}, 
\]
we write 
\[
\int_{\Omega} \Big( \psi _j (D)\big((u_\varepsilon \cdot \nabla) \Delta u_\varepsilon\big)\Big) \cdot \Delta u_\varepsilon \,dx
= - \sum _{l=0}^1 (-1)^{l+1} \int_{\Omega} \big((u_\varepsilon \cdot \nabla)(1+ 2^{-2j-2l} B)^{-1} \Delta u_\varepsilon\big) \cdot \Delta u_\varepsilon \,dx .
\]
We use the identity $\Delta = \nabla \text{div} - \text{rot} \, \text{rot}$ and the divergence-free condition of the term $(1+ 2^{-2j-2l} B)^{-1} \Delta u_\varepsilon$ to write:
\[
\begin{split}
	\Delta u_\varepsilon 
	=& (1+2^{-2j-2l} B) (1+2^{-2j-2l} B)^{-1} \Delta u_ \varepsilon
	\\ 
	= &(1+2^{-2j-2l} \text{rot}\, \text{rot} ) (1+2^{-2j-2l} B)^{-1} \Delta u_ \varepsilon .
\end{split}\]
Substituting this back and applying integration by parts: 
\[
\begin{split}
	&\int_{\Omega} \Big( \psi _j (D)\big((u_\varepsilon \cdot \nabla) \Delta u_\varepsilon\big)\Big) \cdot \Delta u_\varepsilon \,dx
	\\
	&= \sum _{l=0}^1 (-1)^{l+1} 2^{-2j-2l}
	\int_{\Omega} \text{rot} \Big((u_\varepsilon \cdot \nabla)(1+ 2^{-2j-2l} B)^{-1} \Delta u_\varepsilon\Big)
	\cdot \text{rot} \big( (1+ 2^{-2j-2l} B)^{-1} \Delta u_\varepsilon \big) \,dx.
\end{split}
\]
Integration by parts and the boundary condition $u_\varepsilon \cdot \nu = 0$ imply that 
\[
\int_{\Omega} \big((u_\varepsilon \cdot \nabla)U\big) \cdot U \,dx = 0, \quad \text{ with }
U = {\rm rot } \big((1+ 2^{-2j-2l} B)^{-1} \Delta u_\varepsilon \big),
\]
and we obtain: 
\[
\begin{split}
	& \int_{\Omega} \Big( \psi _j (D)\big((u_\varepsilon \cdot \nabla) \Delta u_\varepsilon\big)\Big) \cdot \Delta u_\varepsilon \,dx
	\\
	&=  \sum _{l=0}^1 (-1)^{l+1} 2^{-2j-2l}
	\int_{\Omega}
	\Big\{ \big((\widetilde {\rm rot } \, u_\varepsilon \cdot \nabla) (1+ 2^{-2j-2l} B)^{-1} \Delta u_\varepsilon \big)
	\Big\} \cdot U \,dx, 
\end{split}
\]
where $U = {\rm rot } (1+ 2^{-2j-2l} B)^{-1} \Delta u_\varepsilon$ and the notation $\widetilde{\rm rot } \, u$ is defined as the remainder of the vector calculus identity:
\[
(\widetilde {\rm rot } \, u_\varepsilon \cdot \nabla ) v := {\rm rot } \big( (u_\varepsilon \cdot \nabla) v \big) - (u_\varepsilon \cdot \nabla) {\rm rot } \, v.
\]
Applying the H\"{o}lder inequality to the resulting term: 
\[
\begin{split}
	& \Big| \int_{\Omega} \Big( \psi _j (D)\big((u_\varepsilon \cdot \nabla) \Delta u_\varepsilon\big)\Big) \cdot \Delta u_\varepsilon \,dx \Big|
	\\
	&\leq \sum _{l=0}^1 C 2^{-2j-2l}
	\| \widetilde {\rm rot} \, u_\varepsilon \|_{L^\infty}
	\| \nabla (1+ 2^{-2j-2l} B)^{-1} \Delta u_\varepsilon \|_{L^2}
	\| U \|_{L^2}
	\\
	&\leq C 2^{-2j}  \| \nabla u_\varepsilon \|_{L^\infty} 
	\cdot  \| \nabla (1+ 2^{-2j-2l} B)^{-1} \Delta u_\varepsilon \|_{L^2}  \cdot \| U \|_{L^2}. 
\end{split}
\]
It follows from \eqref{1219-3}  that 
\[
\begin{split}
	\| \nabla(1+ 2^{-2j-2l} B)^{-1} \Delta u_\varepsilon   \|_{L^2},  
	\| U \|_{L^2} 
	\leq  C 2^j \| \psi _j (D)^{\frac{1}{2}} \Delta u_\varepsilon \|_{L^2}, 
\end{split}
\]
which implies that 
\[
\Big| \int_{\Omega} \Big( \psi _j (D)\big((u_\varepsilon \cdot \nabla) \Delta u_\varepsilon\big)\Big) \cdot \Delta u_\varepsilon \,dx \Big|
\leq C \| \nabla u_\varepsilon \|_{L^\infty} 
\| \psi_j (D)^{\frac{1}{2}} \Delta u_\varepsilon \|_{L^2}^2 .
\]

\vskip2mm 

		\noindent 
Step 3. 
We estimate the integral involving the nonlinear remainder term $N(u_{\varepsilon}, u_{\varepsilon})$:
\[
\int _{\Omega } \Big( \psi_j (D) N(u_{\varepsilon}, u_{\varepsilon})\Big) \cdot \Delta u_{\varepsilon} \,dx .
\]
Due to the orthogonality property (see Lemma~\ref{lem:0910-1}) of the pressure term $\nabla {\rm div}\, \big((u_{\varepsilon} \cdot \nabla )u_{\varepsilon}\big)$ with a divergence-free vector $\psi_j (D) \Delta u_{\varepsilon}$, the integral simplifies to:
\[
\int _{\Omega } \Big( \psi_j (D) N(u_{\varepsilon}, u_{\varepsilon})\Big) \cdot \Delta u_{\varepsilon} \,dx
= \int_{\Omega}
\Big\{ (\Delta u_{\varepsilon} \cdot \nabla) u_{\varepsilon} + 2 (\nabla u_{\varepsilon} \cdot \nabla) \nabla u_{\varepsilon} \Big\}
\cdot\psi_j (D)  \Delta u_{\varepsilon} \,dx .
\]
We focus on estimating the term involving $(\Delta u_{\varepsilon} \cdot \nabla) u_{\varepsilon}$, as the second term $2 (\nabla u_{\varepsilon} \cdot \nabla) \nabla u_{\varepsilon}$ can be treated analogously using similar product decomposition techniques.

We use the low-frequency cut-off operator $S_j = \sum _{k \leq j} \phi_k (D)$ to decompose the term $(\Delta u_{\varepsilon} \cdot \nabla) u_{\varepsilon}$: 
\[
\begin{split}
	(\Delta u_{\varepsilon} \cdot \nabla) u_{\varepsilon}
	&= (\Delta u_{\varepsilon} \cdot \nabla) (1-S_j) u_{\varepsilon}
	+ ((1-S_j)\Delta u_{\varepsilon} \cdot \nabla) S_j u_{\varepsilon}
	+ (S_j\Delta u_{\varepsilon} \cdot \nabla) S_j u_{\varepsilon}
	\\
	&= N_1 + N_2 + N_3 .
\end{split}
\]

	Next, we estimate the nonlinear terms. Applying $L^3\times L^6 \times L^2$ estimate, $H^1(\Omega) \hookrightarrow L^3$,  $H^1(\Omega) \hookrightarrow L^6$, and \eqref{1219-2}, the term $N_1$ is estimated as:
\begin{equation}\label{1090-2}
	\begin{split}
		\Big| \int _{\Omega} \Big( \psi _j(D) N_1 \Big) \cdot \Delta u_\varepsilon ~dx \Big| 
		\leq
		& C \| \Delta u_\varepsilon \|_{L^3} \| \nabla (1-S_j) u_\varepsilon\|_{L^6} \| \psi_j(D) \Delta u_\varepsilon\|_{L^2}
		\\
		\leq & C\| u_\varepsilon \|_{H^3(\Omega)}  
		\| \nabla (1-S_j) u_\varepsilon\|_{H^1(\Omega)} \| \psi_j(D) \Delta u_\varepsilon\|_{L^2}
		\\
		\leq & C\| u_\varepsilon \|_{H^3(\Omega)}  
		\|  (1-S_j) u_\varepsilon\|_{H^2(A)} \| \psi_j(D)^{\frac{1}{2}} \Delta u_\varepsilon\|_{L^2}. 
	\end{split}
\end{equation}
For $N_2$, we use $L^2 \times L^\infty \times L^2$ esimate, Lemma~\ref{lem:1222-5} and \eqref{1219-2}: 
\begin{equation}\label{1090-3}
	\begin{split}
	&	\Big| \int _{\Omega} \Big( \psi _j(D) N_2 \Big) \cdot \Delta u_\varepsilon ~dx \Big| 
	\\
		\leq
		& C \| (1-S_j)  \Delta u_\varepsilon\| _{L^2} \| \nabla S_j u_\varepsilon \|_{L^\infty} \| \psi _j(D) \Delta u_\varepsilon \|_{L^2}
		\\
		\leq 
		& C \| (1-S_j)  u_\varepsilon\| _{H^2(A)} \sum _{k \leq j }  2^{\frac{5}{2}j}\| \phi_k(D) u_\varepsilon \|_{L^2} \| \psi _j(D)^{\frac{1}{2}} \Delta u_\varepsilon \|_{L^2}. 
	\end{split}
\end{equation}
		For the term $N_3 = (S_j\Delta u_{\varepsilon} \cdot \nabla) S_j u_{\varepsilon}$, we use the vector calculus identity and boundary conditions.
Since $\psi_j (D) \Delta u_{\varepsilon} = \Delta \psi_j (D) u_{\varepsilon}$ and $\nabla \cdot (\psi_j (D) u_{\varepsilon}) = 0$, we have
\[
\psi_j (D) \Delta u_{\varepsilon} = \text{rot} \, \text{rot} \, \psi_j (D) u_{\varepsilon} .
\]
We apply integration by parts using the curl operator and ${\rm rot \, } \psi_j(D) u_\varepsilon \times \nu = 0 $ on the boundary:
\[
\int _\Omega \Big( \psi_j (D) N_3\Big) \cdot \Delta u_{\varepsilon} \,dx
= \int _{\Omega}
\text{rot} \, N_3 \cdot \text{rot} \, \psi_j (D) u_{\varepsilon} \,dx .
\]
By $L^2 \times L^\infty \times L^2$ estimate: 
\begin{equation}\notag 
	\begin{split}
	&	\Big| \int _{\Omega} \Big( \psi _j(D) N_3 \Big) \cdot \Delta u_\varepsilon ~dx \Big| 
	\\
		\leq 
		&
		 \big( \| \nabla S_j \Delta u_\varepsilon \|_{L^2} \| \nabla S_j u_\varepsilon \|_{L^\infty} 
		+ \| S_j \Delta u _\varepsilon \|_{L^6} \| \nabla^2  S_j u_\varepsilon \|_{L^3}
		\big) 
		\| {\rm rot \, } \psi_j(D) u_\varepsilon \|_{L^2}.
	\end{split}
\end{equation}
On the term $\| {\rm rot \, } \psi_j (D) u_\varepsilon \|_{L^2}$, we estimate 
\[
\| {\rm rot \, } \psi_j (D) u_\varepsilon \|_{L^2} 
\leq C\| B^{\frac{1}{2}} \psi_j (D) u_\varepsilon \|_{L^2}. 
\]
		We notice that 
\[
\begin{split}
	B^{\frac{1}{2}}\psi _j(D)u_\varepsilon  
	=&
	\dfrac{3}{4}\cdot 
	2^{-2j}B^{\frac{1}{2}} (1+2^{-2j-2}B)^{-1} (1+2^{-2j}B)^{-1} B u_\varepsilon
	\\
	=& \dfrac{\sqrt{3}}{2} \cdot 2^{-j } 
	(1+2^{-2j-2}B)^{-\frac{1}{2}} (1+2^{-2j}B)^{-\frac{1}{2}}
	\psi_j(D)^{\frac{1}{2}} \Delta u_\varepsilon ,
\end{split}
\]
and that $(1+2^{-2j-2}B)^{-\frac{1}{2}} (1+2^{-2j}B)^{-\frac{1}{2}}$ 
is a bounded operator uniform on $L^2$ with respect to $j$, 
which yield that 
\[
\| {\rm rot \, } \psi_j(D) u \|_{L^2}
\leq C 2^{-j} \| \psi_j(D)^{\frac{1}{2}} \Delta u_\varepsilon \|_{L^2}, 
\]
Substituting this back and applying Lemma~\ref{lem:1222-5}, we obtain the final estimate for $N_3$:
\begin{equation}\label{1090-4}
	\begin{split}
		& 
		\Big| \int _{\Omega} \Big( \psi _j(D) N_3 \Big) \cdot \Delta u_\varepsilon ~dx \Big| 
			\\
	\leq 
		& C  \Big(\sum_{k \leq j}  2^{3k} \| \phi_k(D) u_\varepsilon \| _{L^2} \Big)    \Big(\sum_{l \leq j}  2^{\frac{5}{2}l} \| \phi_l(D) u_\varepsilon \| _{L^2} \Big)
		\cdot  2^{-j}\|  \psi _j(D)^{\frac{1}{2}} \Delta u_\varepsilon \|_{L^2}. 
	\end{split}
\end{equation}

\noindent 
Step 4. 
	By Steps~1, 2 and the Sobolev embedding $H^2(\Omega) \hookrightarrow L^\infty$, we obtain:
\begin{equation}\label{0113-1}
\begin{split}
		\partial _t \| \psi_j(D)^{\frac{1}{2}} \Delta u_{\varepsilon}  \|_{L^2} ^2 
	\leq 
	&C  \| \nabla u_\varepsilon \|_{L^\infty} \| \psi_j(D)^{\frac{1}{2}} \Delta u_{\varepsilon}  \|_{L^2} ^2 
	+ \Big|  \int _{\Omega}  \Big( \psi_j(D) N(u_\varepsilon , u_\varepsilon) \Big)  \cdot \Delta u_\varepsilon \Big| 
	\\
	\leq 
	&C\| u_\varepsilon \|_{H^3(\Omega)}  \| \psi_j(D)^{\frac{1}{2}} \Delta u_{\varepsilon}  \|_{L^2} ^2 
+ \Big|  \int _{\Omega}  \Big( \psi_j(D) N(u_\varepsilon , u_\varepsilon) \Big)  \cdot \Delta u_\varepsilon \Big|. 
\end{split}
\end{equation}
Assuming $\| \psi_j(D) ^{\frac{1}{2}} \Delta u_\varepsilon \|_{L^2} > 0$, we divide the above by this norm.

Now, fix a natural number $J$ and consider the partial sum over $j \leq J$. Dividing \eqref{0113-1} by $\| \psi_j(D)^{\frac{1}{2}} \Delta u_{\varepsilon} \|_{L^2}$, multiplying by $2^{\frac{1}{2}j}$, integrating in time, and summing over $j \leq J$, the time derivative term becomes:
\[ \sum_{j \leq J} 2^{\frac{1}{2}j}\| \psi_j(D) ^{\frac{1}{2}}\Delta u_\varepsilon (t) \|_{L^2}  
-  \sum_{j \leq J}  2^{\frac{1}{2}j}  \| \psi_j(D)^{\frac{1}{2}} \Delta u_\varepsilon (0) \|_{L^2} . 
\]
For the first term of the right-hand-side of \eqref{0113-1}, we have: 
\[
\begin{split}
	& 
	\sum _{j \leq J} 2^{\frac{1}{2}}\int_0^t  \| u_\varepsilon \|_{H^3(\Omega)} \| \psi_j(D)^{\frac{1}{2}} \Delta u_{\varepsilon}  \|_{L^2}  ~d\tau 
\\
\leq 
&C T^{\frac{1}{2}} \| u_\varepsilon \|_{L^2(0,T; H^3(\Omega))} 
\sup _{t \in [0,T] }\sum_{j \leq J} 2^{\frac{1}{2}j} \| \psi_j(D) \Delta u_\varepsilon (t)\|_{L^2}. 
\end{split}
\]
Applying \eqref{1090-2}, it follows that: 
\[
\begin{split}
	& 
	\sum _{j \leq J} 2^{\frac{1}{2}j}\int_0^t \dfrac{\rm (L.H.S. \, of \, \eqref{1090-2})}{\| \psi_j(D)^{\frac{1}{2}}\Delta u \|_{L^2}} ~d\tau 
	\\
	\leq 
	&C  \| u_\varepsilon \|_{L^2(0,T; H^3(\Omega))} 
	\Big\|  \sum_{j \leq J} 2^{\frac{1}{2}j} \sum_{j < k \leq J } 2^{2k} \| \phi_k(D) u_\varepsilon \|_{L^2}
	+ \| (1-S_J) u_\varepsilon \|_{H^2(A)}
	\Big\|_{L^2(0,T)} 
	\\
	\leq 
	&C  \| u_\varepsilon \|_{L^2(0,T; H^3(\Omega))} 
	\Big( 
	T^{\frac{1}{2}}\Big\|  \sum_{ k \leq J } 2^{\frac{5}{2}k} \| \phi_k(D) u_\varepsilon \|_{L^2}\Big\|_{L^\infty (0,T)}
	+ \| (1-S_J) u_\varepsilon \|_{L^2(0,T; H^2(A))}
	\Big) . 
\end{split}
\]
Similarly, for \eqref{1090-3} and \eqref{1090-4}
\[
\begin{split}
	& 
	\sum _{j \leq J} 2^{\frac{1}{2}j}\int_0^t \dfrac{\rm (L.H.S. \, of \, \eqref{1090-3})}{\| \psi_j(D)^{\frac{1}{2}}\Delta u \|_{L^2}} ~d\tau 
	\\
	\leq 
	&C T^{\frac{1}{2}}
	\Big( 
	T^{\frac{1}{2}}\Big\|  \sum_{ k \leq J } 2^{\frac{5}{2}k} \| \phi_k(D) u_\varepsilon \|_{L^2}\Big\|_{L^\infty (0,T)}
	+ \| (1-S_J) u_\varepsilon \|_{L^2(0,T; H^2(A))}
	\Big) 
	\\
	& 	\cdot \Big\| \sum_{ k \leq J } 2^{\frac{5}{2}k} \| \phi_k(D) u_\varepsilon \|_{L^2}\Big\|_{L^\infty (0,T)}
\end{split}
\]
and 
\[
\begin{split}
	\sum _{j \leq J} 2^{\frac{1}{2}j}\int_0^t \dfrac{\rm (L.H.S. \, of \, \eqref{1090-4})}{\| \psi_j(D)^{\frac{1}{2}}\Delta u \|_{L^2}} ~d\tau 
	\leq 
	& CT	\Big\| \sum_{ k \leq J } 2^{\frac{5}{2}k} \| \phi_k(D) u_\varepsilon \|_{L^2}\Big\|_{L^\infty (0,T)} ^2. 
\end{split}
\]

We introduce the following quantities, 
which arise from the difference between $\phi_j$ and $\psi_j$:   
\[
A_{J}(t) 
:= \sup _{\tau \in (0,t)} \sum_{j \leq J} 2^{\frac{5}{2}j} \| \phi_j(D) u_\varepsilon (\tau) \|_{L^2}, 
\quad 
\widetilde A_{J}(t) 
:= \sup _{\tau \in (0,t)} \sum_{j \leq J} 2^{\frac{5}{2}j} \| \psi_j(D) u_\varepsilon (\tau) \|_{L^2}. 
\]
Note that $u_\varepsilon \in L^2 (0,T ; H^3(\Omega))$, $\sup _{J} \| (1-S_J) u_\varepsilon \|_{L^2(0,T ; H^2(A))} < \infty$, and $A_J(t) \leq C \widetilde A_J (t)$. 
Applying the estimates above to \eqref{0113-1} implies the existence of a constant $C_0$, depending on $u_0$ and $\varepsilon$, such that:
\[
\begin{split}
	\widetilde A_J (t)
	 \leq & \widetilde A_J(0) + C_0 + C_0 (t^{\frac{1}{2} } + t) 
(  A_J(t) + A_J(t)^2)
\\
\leq 
&\widetilde A_J(0) + C_0 + C_0 (t^{\frac{1}{2} } + t) 
( \widetilde  A_J(t) + \widetilde A_J(t)^2).
\end{split}
\]
for all $J $. 
A standard continuity argument then ensures that there exists $T_0 > 0$ such that
\[
\widetilde A_J(t) \leq 2 (\widetilde A_J(0) + C_0) \quad \text{ for all } J  \text{ and } t \in [0,T_0]. 
\]
Taking the limit $J \to \infty$, the equivalence of Besov norms from Lemma~\ref{lem:1222-4} yields:
\[
\| u_\varepsilon (t) \|_{B^{\frac{5}{2}}_{2,1}(A)} \leq 2 (\widetilde A_J(0) + C_0)  
\quad \text{ for all } t \in [0,T_0]. 
\]
We then obtain the regularity $u_\varepsilon$ in $L^\infty (0,T_0; B^{\frac{5}{2}}_{2,1}(A))$. 
\end{pf}

	\begin{prop}\label{prop:1218-1}
		Let $0< \varepsilon< 1$ and $ u_0 \in B^{\frac{5}{2}}_{2,1 } (A)$.
		Then there exists an existence time, $T$, of the solution $u_{\varepsilon} \in C([0,T] , H^1 (A))$ of \eqref{eq:IE}, 
		which depends on $\| u_0 \|_{B^{\frac{5}{2}}_{2,1}(A)}$ and is independent of 
		$\varepsilon$, such that the solution $u_{\varepsilon}$ satisfies the following two properties:
		\begin{enumerate}
			\item Uniform boundedness: 
			$\displaystyle 
			\sup_{0 < \varepsilon < 1} \sum_{j \in \mathbb Z} 2^{\frac{5}{2}j}\| \phi_j(D)u_{\varepsilon}(t) \|_{ L^\infty (0,T ;  L^2)} < \infty , 
			$
			\item Regularity: $u_{\varepsilon} \in C([0,T], B^{\frac{5}{2}}_{2,1}(A))$. 
		\end{enumerate}
	\end{prop}

\noindent {\bf Proof. }	 
Step 1. 
	The existence of a unique solution $u_\varepsilon \in C([0,T], H^1(A)) \cap L^\infty (0,T ; B^{\frac{5}{2}}_{2,1}(A))$ is already established 
	in Proposition~\ref{prop:1225-1}, where $T$ depends on $\varepsilon$. 
	Starting from \eqref{0113-1}, we derive several inequalities in this step modify the argument used in the proof of Proposition~\ref{prop:1225-1}. 

The first inequality of \eqref{0113-1} combined with Lemma~\ref{lem:1222-5} yields: 
\[
		\partial _t \| \psi_j(D)^{\frac{1}{2}} \Delta u_{\varepsilon}  \|_{L^2} ^2 
\leq 
C  \| u_\varepsilon \| _{B^{\frac{5}{2}}_{2,1}(A)} 
    \| \psi_j(D)^{\frac{1}{2}} \Delta u_{\varepsilon}  \|_{L^2} ^2 
+ \Big|  \int _{\Omega}  \Big( \psi_j(D) N(u_\varepsilon , u_\varepsilon) \Big)  \cdot \Delta u_\varepsilon \Big| .
\]
We estimate the integral involving the nonlinear remainder term $N(u_{\varepsilon}, u_{\varepsilon})$. 
For $N_1$, we modify \eqref{1090-2} by using the inequality $\| \Delta u_\varepsilon \|_{L^3} \leq  C \| u_\varepsilon \|_{B^{\frac{5}{2}}_{2,1}(A)}$, ensured by Lemma~\ref{lem:1222-5}. 
We then have 
\[
\begin{split}
	\Big| \int _\Omega \Big( \psi_j (D) N_1\Big) \cdot \Delta u_{\varepsilon} \,dx \Big|
	&\leq \| u_{\varepsilon} \|_{B^{\frac{5}{2}}_{2,1}} \|  (1-S_j) u_{\varepsilon} \|_{H^2(A)} \| \psi_j(D)^{\frac{1}{2}}\Delta u_{\varepsilon} \|_{L^2}
	\\
	&\leq C \| u_{\varepsilon} \|_{B^{\frac{5}{2}}_{2,1}}
	\sum _{ k > j} 2^{2k }\|\phi_k(D)u_{\varepsilon} \|_{L^2}
	\| \psi_j (D)^{\frac{1}{2}} \Delta u_{\varepsilon} \|_{L^2}.
\end{split}
\]
Similarly, for \eqref{1090-3} and \eqref{1090-4} 
\[
\begin{split}
	\Big| \int _\Omega \Big( \psi_j (D) N_2\Big) \cdot \Delta u_{\varepsilon} \,dx \Big|
	\leq 
	& C \left( \sum _{k > j} 2^{2k}\| \phi_k(D) u_{\varepsilon} \|_{L^2} \right)
	\| u_\varepsilon \|_{B^{\frac{5}{2}}_{2,1}(A)}
	\| \psi_j(D)^{\frac{1}{2}}\Delta u_{\varepsilon} \|_{L^2} 
\end{split}
\]
and 
		\[
\begin{split}
	\Big|\int _\Omega \Big( \psi_j (D) N_3\Big) \cdot \Delta u_{\varepsilon} \,dx\Big|
	&\leq C 2^{-j}
	\left( \sum _{k \leq j} 2^{3k} \| \phi_k(D)u_{\varepsilon} \|_{L^2} \right)
	\| u_{\varepsilon} \|_{B^{\frac{5}{2}}_{2,1}(A)}
	\| \psi_j(D)^{\frac{1}{2}} \Delta u_{\varepsilon} \|_{L^2} .
\end{split}
\]
		The analogous term $2 (\nabla u_{\varepsilon} \cdot \nabla) \nabla u_{\varepsilon}$ is treated similarly and yields a comparable estimate.

\vskip2mm 

\noindent 
Step 2. 
		By combining the estimates obtained in the previous step, the energy inequality for the localized equation is given by:
\[
\begin{split}
	\dfrac{1}{2} \partial _t
	\| \psi_j(D)^{\frac{1}{2}} \Delta u_{\varepsilon}
	\|_{L^2}^2
	\leq
	&
	C \| u_\varepsilon \|_{B^{\frac{5}{2}}_{2,1}(A)}
	\| \psi_j (D)^{\frac{1}{2}} \Delta u_\varepsilon \|_{L^2}^2
	\\
	&   +
	C \| u_\varepsilon \|_{B^{\frac{5}{2}}_{2,1}(A)}
	\sum _{ k > j} 2^{2k }\|\phi_k(D)u_\varepsilon \|_{L^2}
	\| \psi_j (D)^{\frac{1}{2} }\Delta u_\varepsilon  \|_{L^2}
	\\
	&
	+C 2^{-j}
	\Big( \sum _{k \leq j} 2^{3k} \| \phi_k(D)u_\varepsilon \|_{L^2} \Big)
	\| u_\varepsilon \|_{B^{\frac{5}{2}}_{2,1}(A)}
	\| \psi_j(D)^{\frac{1}{2}} \Delta u_\varepsilon \|_{L^2}.
\end{split}
\]
Dividing the inequality by $\| \psi_j(D)^{\frac{1}{2}} \Delta u_{\varepsilon} \|_{L^2}$, we obtain the differential inequality for the $L^2$-norm:
\begin{equation}\label{1211-1}
	\begin{split}
		\partial _t
		\| \psi_j(D)^{\frac{1}{2}} \Delta u_{\varepsilon}
		\|_{L^2}
		\leq
		&
		C \| u_\varepsilon \|_{B^{\frac{5}{2}}_{2,1}(A)}
		\| \psi_j (D)^{\frac{1}{2}} \Delta u_\varepsilon \|_{L^2}
		+
		C \| u_\varepsilon \|_{B^{\frac{5}{2}}_{2,1}(A)}
		\sum _{ k > j} 2^{2k }\|\phi_k(D)u_\varepsilon \|_{L^2}
		\\
		&
		+C 2^{-j}
		\Big( \sum _{k \leq j} 2^{3k} \| \phi_k(D)u_\varepsilon \|_{L^2} \Big)
		\| u_\varepsilon \|_{B^{\frac{5}{2}}_{2,1}(A)} .
	\end{split}
\end{equation}
By integrating this inequality in time from $0$ to $t$, multiplying by $2^{\frac{1}{2}j}$, and then summing over all $j \in \mathbb Z$, we arrive at the following estimate:
\[
\begin{split}
	\sum _{j \in \mathbb Z} 2^{\frac{1}{2}j} \| \psi _j (D) ^{\frac{1}{2} } \Delta u_\varepsilon (t)\|_{L^2}
	\leq 
	&\sum _{j \in \mathbb Z} 2^{\frac{1}{2}j} \| \psi _j (D) ^{\frac{1}{2} } \Delta u_\varepsilon (0)\|_{L^2}
	\\
	& + C \int_0^t \| u_\varepsilon \|_{B^{\frac{5}{2}}_{2,1}(A)}
	\sum _{j \in \mathbb Z} 2^{\frac{1}{2}j}\| \psi _j(D)^{\frac{1}{2}} \Delta u_\varepsilon \|_{L^2} \,d\tau
	\\
	& + C \int_0^t 		\| u_\varepsilon \|_{B^{\frac{5}{2}}_{2,1}(A)}
	\sum _{ k \in \mathbb Z} 2^{\frac{5}{2}k }\|\phi_k(D)u_\varepsilon \|_{L^2}
	d\tau . 
\end{split}
\]
Recognizing that the summations above give the $B^{\frac{5}{2}}_{2,1}(A)$ norm (see  Lemma~\ref{lem:1222-4}),  
the estimate simplifies to:
\begin{equation}\label{1211-2}
	\| u_{\varepsilon} (t) \|_{B^{\frac{5}{2}} _{2,1}(A)}
	\leq C \| u_0 \|_{B^{\frac{5}{2}}_{2,1}(A)} + C T \| u_{\varepsilon} \|_{L^\infty (0,T ; B^{\frac{5}{2}}_{2,1}(A))} ^2 ,
\end{equation}
where the constant $C$ is independent of $\varepsilon$.

Moreover, we discuss the independency of the existence time $T$ with respect to $\varepsilon$. 
We know from the proof of Proposition~\ref{prop:1128-1} that we can construct solutions as long as the $H^1(A)$ norm is bounded. 
The estimate \eqref{1211-2}, coupled with 
$B^\frac{5}{2}_{2,1}(A) \hookrightarrow H^1(A)$ and a standard continuity argument implies that the existence time $T$ can be chosen independent of $\varepsilon$, only depending on the initial data: 
\[
T \geq \frac{C}{\| u_0 \|_{B^{\frac{5}{2}}_{2,1}(A)}}, 
\]
with $C$ independent of $\varepsilon$.

\vskip2mm 

\noindent 
Step 3. 
We prove $u \in C([0,T], B^\frac{5}{2}_{2,1} (A))$. 
Integrating \eqref{1211-1}, taking the supremum with respect to the time variable, and summing over $j \in \mathbb Z$: 
\[
\begin{split}
	\sum _{j \in \mathbb Z} 2^{\frac{1}{2}j} \| \psi _j (D) ^{\frac{1}{2} } \Delta u_\varepsilon (t)\|_{L^\infty (0,T;L^2)}
	\leq 
	&C \| u_0 \|_{B^{\frac{5}{2}}_{2,1}(A)}
	\\
	& + C \int_0^T \| u_\varepsilon \|_{B^{\frac{5}{2}}_{2,1}(A)}
	\sum _{j \in \mathbb Z} 2^{\frac{1}{2}j}\| \psi _j(D)^{\frac{1}{2}} \Delta u_\varepsilon \|_{L^2} \,d\tau
	\\
	& + C \int_0^T		\| u_\varepsilon \|_{B^{\frac{5}{2}}_{2,1}(A)}
	\sum _{ k \in \mathbb Z} 2^{\frac{5}{2}k }\|\phi_k(D)u_\varepsilon \|_{L^2}
	d\tau 
	\\
	\leq
	& C \| u_0 \|_{B^{\frac{5}{2}}_{2,1}(A)} + C T \| u_{\varepsilon} \|_{L^\infty (0,T ; B^{\frac{5}{2}}_{2,1}(A))} ^2 < \infty.
\end{split}
\]
We know  $u \in C([0,T], H^1 (A))$, and  Lemma~\ref{lem:1212-1} implies $u \in C([0,T], B^{\frac{5}{2}}_{2,1}(A))$. 
\hfill $\Box$

\section{Proof of Theorem~\ref{thm:1}}

Let $\varepsilon > 0$ and $u_0 \in B^{\frac{5}{2}}_{2,1}(A)$. The existence of a solution of the integral equation \eqref{eq:IE} is ensured by Proposition~\ref{prop:1128-1}, and its regularity follows from Propositions~\ref{prop:1203-1} and \ref{prop:1218-1}. Specifically, there exists a unique solution $u_\varepsilon$ satisfying 
\[
\begin{split}
	& u_\varepsilon \in C([0,T] , B^{\frac{5}{2}}_{2,1}(A)) \cap C^1 ([0,T], H^1(A)) \cap L^2 (0,T; H^3 (\Omega)), \\
	& \partial _t u_\varepsilon \in L^2_{\rm loc} ((0,T], H^2 (A)),
\end{split}
\]
and the equation
\[ 
\partial _t u _\varepsilon - \varepsilon \Delta u _\varepsilon + \mathbb P \Big( (u_\varepsilon \cdot \nabla) u_\varepsilon \Big) = 0 .
\]
Since the solution $u_\varepsilon$ satisfies a sufficient regularity, 
for any test function $\varphi \in C([0,T], H^2(A)) \cap C^1 ((0,T), L^2_\sigma)$ with ${\rm supp \, } \varphi \subset (0,T) \times \overline{\Omega}$, the following integral identity holds:
\begin{equation}\label{1225-2}
	\int _0^T \int _{\Omega} \Big\{ -u_\varepsilon \cdot \partial _t \varphi + \varepsilon u_\varepsilon \cdot \Delta \varphi - \Big( \big( u_{\varepsilon} \cdot \nabla \big) \varphi \Big) \cdot u_\varepsilon \Big\} \,dx dt = 0.
\end{equation}

Compactness property in Lemma~\ref{lem:0113-2} and the uniform boundedness established in Proposition~\ref{prop:1218-1} imply  that, as $\varepsilon \to 0$, a subsequence converges to a limit $u$ in the weak-$*$ topology of $\widetilde L^\infty (0,T ; B^{\frac{5}{2}}_{2,1}(A))$, 
where $\widetilde L^\infty (0,T ; B^{\frac{5}{2}}_{2,1}(A))$ is a Chemin-Lernar space (see \cite{BaChDa_2011}) equipped with the norm defined as follows. 
\[ \| u_\varepsilon \|_{\widetilde L^\infty (0,T B^{\frac{5}{2}}_{2,1}(A))} 
:= \sum _{ j \in \mathbb Z} 2^{\frac{5}{2}j} \| \phi_j(D) u_\varepsilon \|_{L^\infty (0,T; L^2)}.
\]
Noting the equivalence of the norm with replacing $\phi_j$ by $\psi _j$ analogously to Lemma~\ref{lem:1222-4}, 
we get the uniform bounds of the spacial derivative. 
\[ \sup_{0 < \varepsilon < 1} \| \nabla u_\varepsilon \|_{L^\infty (0,T ; L^2)} < \infty . 
\]
We also observe the uniform boundedness of the time derivative with respect to $0 < \varepsilon < 1$. Indeed, 
the uniformity follows from 
\[
\begin{split}
	\| \partial _t u_\varepsilon \|_{L^\infty (0,T; L^2)} 
	&\leq \varepsilon \| \Delta u_\varepsilon \|_{L^\infty (0,T; L^2)} + \big\| \mathbb P \big((u_\varepsilon \cdot \nabla) u_\varepsilon \big) \big\|_{L^\infty (0,T; L^2)} \\
	&\leq C \varepsilon \| u_\varepsilon \|_{L^\infty (0,T ; B^\frac{5}{2}_{2,1}(A))} + C \| u_\varepsilon \|_{L^\infty (0,T ; B^\frac{5}{2}_{2,1}(A))}^2.
\end{split}
\]
The uniform boundedness of both time and spatial derivatives  enables us to extract a subsequence that converges to $u$  in $L^2 ((0,T) \times \Omega)$.   Passing to the limit in \eqref{1225-2}, we obtain
\[
\int _0^T \int _{\Omega} \Big\{ -u \cdot \partial _t \varphi - \Big( \big( u \cdot \nabla \big) \varphi \Big) \cdot u \Big\} \,dx dt = 0,
\]
which implies
\[
\int _0^T \int _{\Omega} \Big\{ -u \cdot \partial _t \varphi + \Big( \big( u \cdot \nabla \big) u \Big) \cdot \varphi \Big\} \,dx dt = 0.
\]
The regularity $u \in \widetilde  L^\infty (0,T ; B^{\frac{5}{2}}_{2,1}(A)) \subset  L^\infty (0,T ; B^{\frac{5}{2}}_{2,1}(A))$ ensures that the nonlinear term $(u \cdot \nabla ) u$ belongs to $L^\infty (0,T ; W^{1,3}(\Omega))$. Consequently, the time derivative $\partial _t u$ exists as an element of $L^\infty (0,T ; W^{1,3} (\Omega))$, and we have
\[
\partial _t u + \mathbb P \big( (u \cdot \nabla) u \big) = 0 \quad \text{ in } W^{1,3}(\Omega) \text{ for almost every } t \in (0,T).
\]

Next, we prove the uniqueness of the solution in $L^\infty (0,T; B^\frac{5}{2}_{2,1}(A))$. Let $u$ and $v$ be two solutions with the same initial data. Following the same justification for the time derivatives as above, we consider the energy estimate for the difference $u-v$:
\[
\frac{1}{2} \partial _t \| u-v \|_{L^2}^2 
= \left| \int _{\Omega} \Big( \big((u-v) \cdot \nabla \big) u \Big) \cdot (u-v) \,dx \right|
\leq \| \nabla u \|_{L^\infty} \| u-v \|_{L^2}^2.
\]
Since $\| \nabla u \|_{L^\infty} \leq C \| u \|_{B^{\frac{5}{2}}_{2,1}(A)}$ and $u \in L^\infty (0,T ; B^{\frac{5}{2}}_{2,1}(A))$, the Gronwall inequality yields $u - v = 0$.

It remains to establish the time continuity of the solution. 
By Lemma~\ref{lem:0113-2}, we can show that $u$ satisfies
\[
\sum _{j \in \mathbb Z} 2^{\frac{1}{2}j} \| \psi _j(D) ^{\frac{1}{2}} \Delta u \|_{L^\infty (0,T; L^2)} < \infty .
\]
Combining this with $\partial _t u \in L^\infty (0,T ; W^{1,3}(\Omega))$ obtained before, we have $u \in C([0,T], L^2)$. Finally, Lemma~\ref{lem:1212-1} ensures the time continuity in $B^{\frac{5}{2}}_{2,1}(A)$, and thus $u \in C([0,T], B^{\frac{5}{2}}_{2,1}(A))$.

	\appendix

	\section{Compact Embedding and Time Continuity}

\begin{lem}\label{lem:0113-2}
	Let $s \in \mathbb R$. 
Suppose that  $ \{u_N \}_{N=1}^\infty $ is a bounded sequence in the Chemin-Lerner space $\widetilde L^\infty (0,T ; B^s_{2,1}(A))$, i.e., 
\[
\sup_{N\in \mathbb N} \sum _{j \in \mathbb Z} 2^{sj} \| \phi_j(D) u_N \|_{L^\infty (0,T ; L^2)}< \infty .
\]
Then a subsequence exists such that it converges in the weak-$*$ topology of  $\widetilde L^\infty (0,T ; B^s_{2,1}(A))$. 
	
\end{lem}

\begin{pf}
	We denote by $c_0(\mathbb{Z}; L^1(0, T; L^2))$ the space of sequences $\{ f_j \}_{j \in \mathbb{Z}}$ of measurable functions on $(0, T) \times \Omega$ such that 
	\[
	\sup_{j \in \mathbb{Z}} \| f_j \|_{L^1(0, T; L^2)} < \infty \quad \text{and} \quad \lim_{|j| \to \infty} \| f_j \|_{L^1(0, T; L^2)} = 0.
	\]
	It is well known that $c_0(\mathbb{Z}; L^1(0, T; L^2))$ is separable. 
	
	We now consider the mapping
	\[
	\widetilde{L}^\infty(0, T; B^s_{2,1}(A)) \ni u \quad \mapsto \quad \{ 2^{sj} \phi_j(D) u \}_{j \in \mathbb{Z}} \in \ell^1(\mathbb{Z}; L^\infty(0, T; L^2)).
	\]
	This map is an isometry, and its range is a closed subspace of $\ell^1(\mathbb{Z}; L^\infty(0, T; L^2))$. Note that the space $\widetilde{L}^\infty(0, T; B^s_{2,1}(A))$ can be identified with a closed subspace of the dual of $\ell^1(\mathbb{Z}; L^1(0, T; L^2))$. More precisely, its predual is a quotient space (or closed subspace) related to $c_0(\mathbb{Z}; L^1(0, T; L^2))$. 
	For any bounded sequence in $\widetilde{L}^\infty(0, T; B^s_{2,1}(A))$, the Banach-Alaoglu theorem, combined with the separability of the predual, ensures the existence of a subsequence that converges in the weak-$*$ topology of $\widetilde{L}^\infty(0, T; B^s_{2,1}(A))$.
\end{pf}

\begin{lem}\label{lem:1212-1}
Let  $s \geq 0$. 
Suppose that $u \in C([0,T], L^2) \cap \widetilde L^\infty (0,T ; B^s_{2,1}(A))$. 
Then $u \in C([0,T], B^s_{2,1}(A))$. 
\end{lem}
\begin{pf} 
	Let $u$ satisfy the assumptions in Lemma~\ref{lem:1212-1}. Suppose $t_0 \in [0, T]$; we consider the continuity of $u(t)$ at $t = t_0$. 
	Let $\varepsilon > 0$ be an arbitrary constant. We first choose $J_0 \in \mathbb{N}$ large enough so that 
	\[
	\sum_{|j| > J_0} 2^{sj} \sup_{t \in [0,T]}    \| \phi_j(D) u(t) \|_{L^2} < \varepsilon.
	\]
	This choice implies that
	\[
	\sum_{|j| > J_0} 2^{sj} \| \phi_j(D) (u(t) - u(t_0)) \|_{L^2} \leq 2 \sup_{t \in [0,T]} \sum_{|j| > J_0} 2^{sj} \| \phi_j(D) u(t) \|_{L^2} < 2\varepsilon
	\]
	for all $t \in [0, T]$. On the other hand, the continuity of $u$ in $C([0, T]; L^2)$ and the fact that the sum over $|j| \leq J_0$ is a finite sum ensure that there exists $\delta > 0$ such that 
	\[
	\sum_{|j| \leq J_0} 2^{sj} \| \phi_j(D) (u(t) - u(t_0)) \|_{L^2} < \varepsilon
	\]
	provided that $|t - t_0| < \delta$. Combining these estimates, we conclude that $u(t) \to u(t_0)$ in $B^s_{2,1}(A)$ as $t \to t_0$.
\end{pf}
	
\section*{Acknowledgement}
This work was supported by JSPS 	
Grant-in-Aid for Scientific Research (A) JP21H04433. 
\par
\bigskip
{\bf Declaration of interests} The authors declare no conflicts of interest.
\par
{\bf Data availability statement} 
Data sharing is not applicable to this article as no data were created or analyzed in this study.

\end{document}